\documentclass[12pt]{amsart}
\usepackage{amscd,amssymb,amsmath, array,latexsym,amsbsy}
\usepackage[all]{xy}
\usepackage{hyperref}

\newcommand{\go}{G^{(0)}}
\newcommand{\cc}{C_{c}}

\newcommand{\lemref}[1]{Lemma~\ref{#1}}

\newcommand{\cs}{\ensuremath{C^{*}}}
\def\supp{\operatorname{supp}}
\def\tr{\operatorname{tr}}

\def\range{\operatorname{range}}

\def\Rep{\operatorname{Rep}}

\def\H{\mathcal{H}} 
\def\O{\mathcal{O}}
\def\G{\mathcal{G}}
\def\cA{\mathcal{A}}
\def\cD{\mathcal{D}}
\def\cR{\mathcal{R}} 
\def\C{\mathbb{C}}
\def\T{\mathbb{T}}
\def\Z{\mathbb{Z}}
\def\N{\mathbb{N}}
\def\P{\mathbb{P}}
\newtheorem{thm}{Theorem}[section]

\newtheorem{prop}[thm]{Proposition}

\newtheorem{lemma}[thm]{Lemma}
\theoremstyle{definition}

\newtheorem{remark}[thm]{Remark}

\newtheorem{example}[thm]{Example}
\numberwithin{equation}{section}

\newcommand\set[1]{\{\,#1\,\}}

\usepackage{color}

\begin{document}

\author[Clark]{Lisa Orloff Clark}
\address{Dept of Mathematical Sciences
\\Susquehanna University
\\Selinsgrove,  PA 17870
\\USA
} \email{clarklisa@susqu.edu}

\author[an Huef]{Astrid an Huef}
\address{Department of Mathematics and Statistics\\
University of Otago\\PO Box 56\\ Dunedin 9054\\
New Zealand}
\email{astrid@maths.otago.ac.nz}

\title{\boldmath{The representation theory of $C^*$-algebras associated to
groupoids}}

\date{12 July 2010}
\thanks{This research was supported by the Australian Research
Council and the Association for Women in Mathematics.}
 \maketitle

\begin{abstract}Let $E$ be a second-countable, locally compact,
Hausdorff groupoid equipped with an action of $\T$ such that
$G:=E/\T$  is a principal groupoid with Haar system $\lambda$.  The
twisted groupoid $\cs$-algebra $\cs(E;G,\lambda)$ is a quotient of
the $C^*$-algebra of $E$ obtained by completing the space of
$\T$-equivariant functions on $E$.
 We show that $\cs(E;G,\lambda)$ is postliminal if and only if
the orbit space of $G$ is $T_0$ and that $\cs(E;G, \lambda)$ is
liminal if and only if the orbit space is $T_1$. We also show that
$\cs(E;G, \lambda)$ has bounded trace if and only if $G$ is
integrable and that $\cs(E;G, \lambda)$ is a Fell algebra if and
only if $G$ is Cartan.

Let $\G$ be a second-countable, locally compact, Hausdorff groupoid
with Haar system $\lambda$ and continuously varying, abelian
isotropy groups. Let $\cA$ be the isotropy groupoid and $\cR :=
\G/\cA$. Using the results about twisted groupoid $C^*$-algebras, we
show that the $\cs$-algebra $\cs(\G, \lambda)$ has bounded trace if
and only if $\cR$ is integrable and that $\cs(\G, \lambda)$ is a
Fell algebra if and only if $\cR$ is Cartan.  We illustrate our
theorems with examples of groupoids associated to directed graphs.
\end{abstract}

\section{Introduction}

Let $H$ be a locally compact, Hausdorff group acting continuously on
a locally compact, Hausdorff space X.   When the orbit space $X/H$
is reasonable, for example  if $X/H$ is $T_0$, then every
irreducible representation of the transformation-group $C^*$-algebra
$C_0(X)\rtimes H$ is induced from an irreducible representation of
an isotropy subgroup $S_x=\{h\in H: h\cdot x=x\}$.  In particular,
if the action of $H$ on $X$ is free then the spectrum of
$C_0(X)\rtimes H$ is homeomorphic to the orbit space by
\cite{green77}, or if $H$ is abelian then the spectrum of
$C_0(X)\rtimes H$ is homeomorphic to a quotient of $(X/H)\times\hat
H$ by \cite{dana-ccr}.  Many of the postliminal (Type \textrm{I})
properties of the transformation-group $C^*$-algebra can be deduced
from the dynamics of the transformation group $(H, X)$.  For
example, $C_0(X)\rtimes H$ is postliminal if and only if the orbit
space is $T_0$ and all the isotropy subgroups are postliminal
\cite{gootman}.  There are many more results of this nature in the
literature:   \cite{green77, wiljfa81,   echter} investigate when
$C_0(X)\rtimes H$ has continuous trace, \cite{aH1, aH2, AaH} when
$C_0(X)\rtimes H$ is a Fell algebra or has bounded trace, and
\cite{dana-ccr} when $C_0(X)\rtimes H$  is liminal.  Usually the
results are first proved for free actions and then generalized to
non-free actions; but even when the isotropy groups are abelian the
level of technical difficulty is much greater, and to get general
results assumptions on the isotropy subgroups (for example,
amenability or that they  vary continuously) often seem unavoidable.
The theorems above have served as a template for establishing
similar theorems for  the $C^*$-algebras of directed graphs
\cite{ephrem, Robbie} and the $C^*$-algebras of groupoids
\cite{MW90, MW92,MRW96,C,C2,CaH}.

Let $\G$ be a locally compact, Hausdorff groupoid with abelian
isotropy subgroups,  and let $\cA$ be the isotropy groupoid. The
main theorem of \cite{MRW96} says that $C^*(\G)$ has continuous
trace if and only if the isotropy groups vary continuously and
$\G/\cA$ is a proper groupoid.  The proof strategy, quickly, is to
show that all the irreducible representations of $C^*(\G)$ are
induced, and to use the dual isotropy groupoid $\hat\cA$ to
construct a $\T$-groupoid whose associated twisted groupoid
$C^*$-algebra is isomorphic to $C^*(\G)$.  Then the characterization
of when twisted groupoid $C^*$-algebras have continuous trace from
\cite{MW92} completes the proof.

In this paper we generalize first the results from \cite{MW92}  to characterize when
 a twisted groupoid $C^*$-algebra has bounded trace or is a Fell algebra (Theorems~\ref{thm-hard}
  and~\ref{thm-Fell}), and
   second, the  results  from \cite{MRW96} to characterize when a groupoid with continuously varying,
   abelian isotropy groups has bounded trace  or is a Fell algebra (Theorems~\ref{thm-end} and~\ref{thm-Cartan2}).
To  do this we  had to deal with non-Hausdorff spectra, which led to
a  sharpening of \cite[Proposition~3.3]{MW92} and
\cite[Proposition~4.5]{MRW96} (see Theorem~\ref{MWthm-better} and
Proposition~\ref{isomo}).  Theorem~\ref{MWthm-better} says that when
the orbit space of the groupoid is $T_0$, then the spectrum of the
twisted groupoid $C^*$-algebra is homeomorphic to the orbit space
and Proposition~\ref{isomo} establishes the isomorphism of $C^*(\G)$
with the twisted groupoid $C^*$-algebra of \cite{MRW96} mentioned
above under weaker hypothesis. Finally we illustrate our theorems
with examples of groupoids associated to directed graphs.

\section{Preliminaries}
Let $G$ be a  locally compact, Hausdorff groupoid. We denote
the unit space of $G$  by $G^{(0)}$, the range and source maps
$r,s:G\to G^{(0)}$ are $r(\gamma)=\gamma\gamma^{-1}$ and
$s(\gamma)=\gamma^{-1}\gamma$, respectively, and the set of
composable pairs by $G^{(2)}$.  Recall that $G$ is \emph{principal}
if the map $\gamma\mapsto (r(\gamma),s(\gamma))$ is injective.

Let $N\subseteq G^{(0)}$.  The \emph{saturation} of $N$ is
$r(s^{-1}(N))=s(r^{-1}(N))$, and if $N=r(s^{-1}(N))$ then we say
that $N$ is \emph{saturated}. We define the \emph{restriction of $G$
to $N$} to be $G|_N:=\{\gamma\in G:s(\gamma)\in N\text{\ and\
}r(\gamma)\in N\}$.  The latter is not to be confused with $G_N
:=\{\gamma \in G : s(\gamma) \in N\}$. If $u\in G^{(0)}$, we call
the saturation of $\{u\}$ the \emph{orbit} of $u\in G^{(0)}$ and
denote it by $[u]$; we also  write  $G_u$ instead of $G_{\{u\}}$.

\subsection{$\T$-groupoids} A \emph{$\T$-groupoid} $E$ is a topological groupoid $E$ with a continuous free action of the circle
group $\T$ on $E$ such that
\begin{enumerate}
\item if $(\gamma_1, \gamma_2) \in E^{(2)}$ and $s,t \in
\T$ then
\[
(s\gamma_1, t\gamma_2) \in E^{(2)}\quad \text{and}\quad (s \gamma_1)(t \gamma_2) = (ts) (\gamma_1 \gamma_2);
\]
\item $G:=E/\T$ is a principal groupoid.
\end{enumerate}
In what follows, we will always assume that $E$ is
second-countable, locally compact and Hausdorff. Note that the
composibility condition (1) implies that $s(\gamma)=s(t\cdot\gamma)$
and $r(\gamma)=r(t\cdot\gamma)$ for all $\gamma\in E$ and $t\in \T$;
in particular, $E^{(0)}=G^{(0)}$. That $G$ is principal implies that
there is an exact sequence
\begin{equation}\label{ses}
E^{(0)} \longrightarrow E^{(0)} \times \T \overset{i}{\longrightarrow} E
\overset{j} {\longrightarrow} G \longrightarrow E^{(0)}
\end{equation}
where  $i$ is the homeomorphism  $i(u,t) = t\cdot u$   onto a closed
subgroupoid and $j$ is the quotient map.  Conversely, starting with
a sequence~\eqref{ses}, there  is a free action of $\T$ on $E$
defined by $t\cdot\gamma=i(r(\gamma), t)\gamma$, and the quotient
$E/\T$ can be identified with $G$.

\begin{remark} Since $\T$ is compact, $G=E/\T$ is Hausdorff, and since $\T$ is a compact Lie group,
$E$ is a locally trivial bundle over $G$ by \cite[Proposition~4.65 and Hooptedoodle~4.68]{tfb}.
That the sequence \eqref{ses} is exact is equivalent to: every $\gamma$ in the isotropy groupoid
\[\cA:=\{\gamma\in E:s(\gamma)=r(\gamma)\}\] can be written as $t\cdot u$ for some $t\in\T$ and $u\in E^{(0)}$.
Thus our $\T$-groupoid is what is called a ``proper $\T$-groupoid''
in \cite[Definition~2.2]{K}. But since  we do not assume that
$G=E/\T$ is \'etale, $E$ is not a ``twist'' in the sense of
\cite[Definition~2.4]{K}; $\T$-groupoids are more
general. In particular, our assumption
that $E$ is $\T$-groupoid such that $G=E/\T$ is a principal groupoid
puts us in the situation of \cite{MW92}.
\end{remark}


\subsection*{Construction of the twisted groupoid $C^*$-algebra}
We briefly outline the construction of the twisted groupoid
$C^*$-algebra from \cite{MW92}. Let $E$  be a $\T$-groupoid over a
principal groupoid $G$ equipped with a left Haar system
$\{\lambda^u:u\in G^{(0)}\}$.  Then there is a left Haar system
$\{\sigma^u:u\in E^{(0)}=G^{(0)}\}$ on $E$ characterized  by
\begin{equation}\label{eq-Haar}
\int_E f(\alpha)\, d\sigma^u(\alpha)=\int_G\int_\T f(t\cdot\alpha)\, dt\, d\lambda^u(j(\alpha))\quad\text{($f\in C_c(E)$).}
\end{equation}
  A left Haar system $\{\lambda^u:u\in G^{(0)}\}$ gives  a right Haar system $\{\lambda_u:u\in G^{(0)}\}$ via $\lambda_u(E):=\lambda^u(E^{-1})$,
  and we will move freely between the left and right systems when convenient.

The usual  groupoid $C^*$-algebra $C^*(E,\sigma)$ of $E$  is the
$C^*$-algebra which is universal  for continuous nondegenerate $*$-representations $L:C_c(E)\to B(\H_L)$, where
$C_c(E)$ has the inductive limit topology, $B(\H_L)$ the weak operator topology, and $C_c(E)$ is a $*$-algebra via
\begin{equation*}\label{eq-algebraicops}
f*g(\gamma)=\int_E f(\gamma\alpha)g(\alpha^{-1})\,
d\sigma^{s(\gamma)}(\alpha)\quad\text{and}\quad
f^*(\gamma)=\overline{f(\gamma^{-1})}
\end{equation*}
for $f,g\in C_c(E)$.

The \emph{twisted groupoid $C^*$-algebra}  $C^*(E;G,\lambda)$ is a
quotient of $C^*(E,\sigma)$ obtained as follows. Let $C_c(E;G)$ be
the collection of $f\in C_c(E)$ such that
$f(t\cdot\gamma)=tf(\gamma)$. Note that for $f,g\in C_c(E;G)$ and
$\gamma\in E$, the function $\alpha\mapsto
f(\gamma\alpha)g(\alpha^{-1})$  depends only on the class
$j(\alpha)$ of $\alpha$. So we can equip $C_c(E;G)$  with a
*-algebra structure via
\[
f*g(\gamma)=\int_G f(\gamma\alpha)g(\alpha^{-1})\,
d\lambda^{s(\gamma)}(j(\alpha))\quad\left(= \int_E f(\gamma\alpha)g(\alpha^{-1})\,
d\sigma^{s(\gamma)}(\alpha)\right)
\]
and $f^*(\gamma)=\overline{f(\gamma^{-1})}$ for $f,g\in C_c(E;G)$; using \eqref{eq-Haar} it is straightforward
to check that the formulae for $f*g$ in $C_c(E)$ and $C_c(E;G)$ coincide.
 Let $\Rep(E;G)$ be the collection of non-degenerate $*$-representations  $L:C_c(E;G)\to B(\H_L)$ which are
 continuous when $C_c(E;G)$ has the inductive limit topology and $B(\H_L)$ has the weak operator topology.  It follows from
\cite[Proposition~3.5 and Th\`eor\'eme~4.1]{renault-jot} that for
$f\in C_c(E;G)$
\[
\|f\|=\sup\{\|L(f)\|:L\in\Rep(E;G)\}
\]
is finite and defines a pre-$C^*$-norm on $C_c(E;G)$.   The
completion of $C_c(E;G)$ in this norm is the twisted groupoid
$C^*$-algebra $C^*(E;G,\lambda)$. That $C^*(E;G,\lambda)$ is a
quotient of $C^*(E,\sigma)$ follows because $\Rep(E;G)$ is a subset
of the representations considered when constructing $C^*(E,\sigma)$.
By Lemma~3.3 of \cite{renault-jot}, the surjective homomorphism
$\Upsilon:C_c(E)\to C_c(E;G)$ defined by
\begin{equation}\label{defnchi}
\Upsilon(f)(\gamma)=\int_\T f(t\cdot\gamma)\bar t\, dt
\end{equation}
is continuous in the inductive limit topology, and hence extends to
a homomorphism $\Upsilon:C^*(E,\sigma)\to C^*(E;G,\lambda)$ called
the quotient map. The reasons for calling $C^*(E;G,\lambda)$ the
``twisted groupoid $C^*$-algebra'' are outlined in \cite[\S2]{MW92}.

\subsection{Postliminal properties of $\cs$-algebras}  Let $A$ be a $C^*$-algebra and $\hat A$ its spectrum.
If $\pi$ is an irreducible representation of $A$ then we write
$\H_\pi$ for the Hilbert space on which $\pi(A)$ acts. If
$\pi(A)\supseteq K(\H_\pi)$ for every irreducible representation
$\pi$ of $A$, then $A$ is \emph{postliminal}; if $\pi(A)=
K(\H_\pi)$ for every irreducible representation $\pi$ of $A$, then
$A$ is \emph{liminal}. In the literature postliminal and liminal
$C^*$-algebras are also called GCR and CCR $C^*$-algebras,
respectively.  A positive element $b\in A$ is called a
\emph{bounded-trace element} if the map $[\pi] \mapsto \tr(\pi(b))$
is bounded on  $\hat A$. Then $A$ has \emph{bounded trace} if the
ideal consisting of the linear span of bounded-trace elements is
dense in A. An  irreducible representation $\pi$ of $A$ satisfies
\emph{Fell's condition} if there is a positive $a \in A$ and a
neighbourhood $U$ of $[\pi]$ in $\hat A$ such that $\sigma(a)$ is a
rank-one projection whenever $[\sigma] \in U$. If every irreducible
representation of $A$ satisfies Fell's condition then $A$ is a
\emph{Fell algebra}. A Fell algebra  $A$  has Hausdorff spectrum if
and only if  $A$ has continuous trace.  Each of the properties above
are listed in order of reverse containment.


\section{The spectrum of a twisted groupoid $\cs$-algebra}

We start by investigating ideals in $C^*(E;G,\lambda)$ associated to
open saturated subsets of the unit space of $G$.

\begin{lemma}\label{lem-exact}
Suppose that $E$ is a second-countable, locally compact, Hausdorff,
$\T$-groupoid such that $G:=E/\T$  is a principal groupoid with Haar
system $\lambda$.  Let $\sigma$ be the Haar system on $E$ defined at
\eqref{eq-Haar} and $U$ an open saturated subset of $G^{(0)}$
with $F:=G^{(0)}\setminus U$.  Then the short exact sequence
\begin{equation}\label{eq-exactMRW}
0\to C^*(E|_U,\sigma)\stackrel{i}\to C^*(E,\sigma)\stackrel{p}\to C^*(E|_F,\sigma)\to 0
\end{equation}
of \cite[Lemma~2.10]{MRW96} induces a short exact sequence
\begin{equation}\label{eq-exact}
0\to C^*(E|_U; G|_U,\lambda)\stackrel{k}\to C^*(E;G,\lambda)\stackrel{r}\to C^*(E|_F; G|_F,\lambda)\to 0
\end{equation}
such that $k$ is isometric and $\Upsilon\circ i=k\circ\Upsilon$ and
$\Upsilon\circ p=r\circ \Upsilon$.  On  continuous functions the
maps $k$ and $p$ are extension by $0$ and restriction, respectively.
\end{lemma}

\begin{proof}  Note that we write  just $\Upsilon$  for both the homomorphisms $\Upsilon:C^*(E,\sigma)\to C^*(E;G,\lambda)$  and
 $\Upsilon:C^*(E|_U,\sigma)\to C^*(E|_U;G|_U,\lambda)$.   Since
$\Upsilon\circ i=\Upsilon\circ i\circ\Upsilon$ on $C_c(E|_U)$ we
have  $\ker\Upsilon\subseteq\ker(\Upsilon\circ i)$, and hence there
exists a unique homomorphism $k:C^*(E|_U; G|_U,\lambda)\to
C^*(E;G,\lambda)$ such that $\Upsilon\circ i=k\circ\Upsilon$.
Similarly, $\Upsilon\circ p=\Upsilon\circ p\circ\Upsilon$ on
$C_c(E)$, so $\ker \Upsilon\subseteq\ker(\Upsilon\circ p)$, and
hence there exists a unique homomorphism $r: C^*(E;G,\lambda)\to
C^*(E|_F; G|_F,\lambda)$ such that $r\circ\Upsilon=\Upsilon\circ p$.
Note that $r$ is surjective because $p$ and $\Upsilon$ are.

To see that $k$ is isometric, fix a  representation $\pi$ of
$C^*(E|_U; G|_U, \lambda)$. It suffices to see that $\pi$ determines
a representation $\hat \pi$ of $C_c(E;G)$ such that
$\|\pi(f)\|=\|\hat\pi(k(f))\|$ for $f\in C_c(E|_U; G|_U)$; this will
give $\|k(f)\|\geq \|f\|$ and hence $\|k(f) \|=\|f\|$.

By  \cite[Lemma~2.10]{MRW96}, $i$ is an isometric isomorphism of $C^*(E|_U)$ onto an ideal $I$ of $C^*(E)$. Let
$\tilde\pi :C^*(E)\to B(\H_\pi)$ be the canonical extension of   $\pi\circ\Upsilon\circ i^{-1}: I\to B(\H_\pi)$.
Note that, for $g\in C_c(E)$ and $h\in C_c(E|_U)\subseteq C_c(E)$ we have
\begin{align*}
\tilde\pi(\Upsilon(g))\pi\circ\Upsilon\circ i^{-1}(h)
&=\pi\circ\Upsilon\circ i^{-1}(\Upsilon(g)h)=\pi\circ k^{-1}\circ \Upsilon(\Upsilon(g)h)\\
&=\pi\circ k^{-1}\circ \Upsilon(gh)=\pi\circ\Upsilon\circ i^{-1}(gh)\\
&=\tilde\pi(g)\pi\circ\Upsilon\circ i^{-1}(h).
\end{align*}
Thus $\tilde\pi(\Upsilon(g))=\tilde\pi(g)$ and hence $\tilde\pi$ factors through $C^*(E;G, \lambda)$ and gives
a representation $\hat\pi: C^*(E;G, \lambda)\to B(\H_\pi)$ such that $\tilde\pi=\hat\pi\circ\Upsilon$.
Finally, if $f\in C_c(E|_U; G|_U)$ then for all $h\in C_c(E|_U)$ we have
\[
\hat\pi(k(f))\pi\circ\Upsilon\circ i^{-1}(h)=\pi\circ\Upsilon\circ i^{-1}(k(f)h)=\pi\circ\Upsilon(fh)=\pi(f)\pi\circ\Upsilon\circ i^{-1}(h),
\]
and hence $\hat\pi(k(f))=\pi(f)$. Thus $k$ is isometric.

Since $k$ is isometric, the image of $k$ is the completion of $C_c(E|_U; G|_U)$ viewed as functions on
$E$. Since $U$ and $F$ are disjoint  $r(C_c(E|_U ;G|_U))=0$ and hence $\range k\subseteq\ker r$.
Conversely, if $h\in C_c(E)\cap \ker r$ then  $h$ has support in $U$ and hence  is in the range of $i$. Thus $\range k= \ker r$.
\end{proof}

Fix $u \in \go$ and let $\H^0_u$ be the collection of bounded Borel
functions  $f$ on $E$ with compact support in $E_u=s^{-1}(\{u\})$
satisfying $f(t\cdot\gamma)=tf(\gamma)$ for all $t\in\T$ and
$\gamma\in E$. For each $\xi,\eta\in \H_u^{0}$ define
\[
( \xi\mid  \eta )_u = \int_G \xi(\gamma)\overline{\eta(\gamma)} \ d\lambda_u(j(\gamma))\quad\left(= \int_E \xi(\gamma)\overline{\eta(\gamma)} \ d\sigma_u(\gamma)\right)
\]
to get an inner product on $\H_u^0$.  Denote by $\H_u$ the
Hilbert space completion of $\H_u^0$ with respect to this inner product; note that $\H_u$ is a closed subspace of $L^2(E_u,\sigma_u)$.
Moreover, the functions obtained by
restricting elements of $\cc(E;G)$ to $E_u$ form a dense subset of
$\H_u$ (see  \cite[Page~133]{MW92}).

Let $f,\xi\in C_c(E;G)$.  By \cite[\S3]{MW92}, the formula
\begin{equation}\label{eq-Lu}
L^u(f)\xi(\gamma) = f* \xi(\gamma)=\int_G
f(\gamma\alpha)\xi(\alpha^{-1})\,
d\lambda^u(j(\alpha))\end{equation} defines an appropriately
continuous representation $L^u(E;G)=L^u$ of $C_c(E;G)$ on a dense
subspace of $\H_u$, whence $L^u$ extends to a representation $L^u$ of
$C^*(E;G,\lambda)$ on $\H_u$.  By \cite[Lemma~3.2]{MW92}, $L^u$ is
irreducible, and if $[u]=[v]$ then $L^u$ and $L^v$ are unitarily
equivalent.

In Proposition~3.3 of \cite{MW92}, Muhly and Williams prove that if
$C^*(E;G,\lambda)$ has Hausdorff spectrum, then $L:u\mapsto [L^u]$
induces a homeomorphism $\Psi$ from the orbit space $G^{(0)}/G$ onto
the spectrum of $C^*(E;G,\lambda)$; it seems from the application of
\cite[Proposition~3.3]{MW92} in the proof of
\cite[Proposition~4.5]{MRW96} that its authors knew that the proof
goes through using only that $C^*(E;G,\lambda)$ has $T_1$ spectrum
(see \cite[middle of p.~3638]{MRW96} and the  applications of
\cite[Proposition~4.5]{MRW96} in the proof of
\cite[Theorem~1.1]{MRW96}).

The original proof of \cite[Proposition~3.3]{MW92} refers the reader to the proof of \cite[Proposition~25]{MW90}
 to see that $\Psi$ induces a continuous injection; since the notations of \cite{MW92} and \cite{MW90} don't
 match up, we had to carefully go through the details to verify that the Hausdorff condition wasn't needed, and we
 record these details here. The proof that $\Psi$ is open onto its range given in \cite[Proposition~3.3]{MW92} used
 that $C^*(E;G,\lambda)$ has $T_1$ spectrum; our argument below does not require this hypothesis.
We strengthen Proposition~\ref{MWthm} further in Theorem~\ref{MWthm-better} below.

\begin{prop}[Muhly-Williams]
\label{MWthm} Let $E$ be a second-countable, locally compact,
Hausdorff, $\T$-groupoid such that $G:=E/\T$  is a principal
groupoid with Haar system $\lambda$.  For $u\in G^{(0)}$ let $L^u$
be the irreducible representation defined at \eqref{eq-Lu}.
\begin{enumerate}
\item Then the map $u\mapsto [L^u]$ induces a continuous
injection $\Psi: G^{(0)}/G\to C^*(E;G,\lambda)^\wedge$ which is open  onto its range.
\item If $G^{(0)}/G$ is $T_1$ then $\Psi$ is a homeomorphism of $G^{(0)}/G$ onto $C^*(E;G,\lambda)^\wedge$.
\end{enumerate}
\end{prop}
\begin{proof}
(1) We start by showing that $\Psi$ is continuous.   Fix  $\xi, \eta \in \cc(E;G)$ and
$a\in C^*(E;G,\lambda)$. We claim  that the map $x\mapsto ( L^x(a)\xi\,|\eta)_x$ is continuous.  To see why this is so, first consider $f\in \cc(E;G)$ and note that
\[
( L^x(f)\xi\,|\eta)_x=\int_G(f*\xi)(\gamma)\overline{\eta(\gamma)}\, d\lambda_x(j(\gamma)),
\]
where the convolution $f*\xi$ is taking place in $C_c(E;G)$.  Since $(f*\xi)(\gamma)\overline{\eta(\gamma)}$
 has compact support, the continuity of $x\mapsto ( L^x(f)\xi\,|\eta)_x$ follows from the
continuity of the Haar system.  It now follows from an $\epsilon/3$
argument that the map $x\mapsto
( L^x(a)\xi\,|\,\eta)_x$ is also continuous.

Now suppose that $u_n\to u$ in $G^{(0)}$; we will show by way of
contradiction that $[L^{u_n}]\to [L^u]$. Suppose that $[L^{u_n}]$
does not converge to $[L^u]$.  Then there exists a neighbourhood $O$
of $[L^u]$ such that $[L^{u_n}]\notin O$ frequently. By passing to a
subsequence and relabeling  we may assume $[L^{u_n}]\notin O$ for
all $n$.  Let $J$ be the ideal of $C^*(E;G,\lambda)$ such that $O=
\{ \rho \in \cs(E;G,\lambda)^{\wedge} \mid \rho(J) \neq 0\}$. So
there exists $a\in J$ such that   and $L^u(a) \neq 0$ and
$L^{u_n}(a) = 0$ for all $n$. Now choose functions $\xi,\eta \in
\cc(G)$ such that
\begin{equation}
    ( L^u(a)\xi\mid \eta )_u\neq 0 \notag;
\end{equation} but now
\begin{equation}
0= ( L^{u_n}(a)\xi\,|\, \eta)_{u_n}
\rightarrow ( L^u(a)\xi\, |\, \eta)_{u}\neq 0\notag,
\end{equation}
contradicting the continuity of $x\mapsto  ( L^x(a)\xi\mid \eta )_x$.  
So $[L^{u_n}]\to [L^u]$, and it follows that $\Psi$ is continuous.

Next we show that $\Psi$ is injective.  Let $u,v\in G^{(0)}$ and
suppose that $L^u$ and $L^v$ are unitarily equivalent.  We will show
that $[u]=[v]$.  By \cite[Lemma~3.1]{MW92} there is a homomorphism
$R:C_0(G^{(0)})\to M(C^*(E;G,\lambda))$ defined by
\[
(R(\phi)f)(\gamma)=\phi(r(\gamma))f(\gamma)
\]
for $f\in C_c(E;G)$.  In
the proof of \cite[Lemma~3.2]{MW92} Muhly and Williams show that
$L^u$ is unitarily equivalent to a representation
$T^u:C^*(E;G,\lambda)\to B(L^2([u],\mu_{[u]})$, and that
$N_u:=\overline{T^u}\circ R:C_0(G^{(0)})\to B(L^2([u],\mu_{[u]})$
has formula $N_u(\phi)\eta(x)=\phi(x)\eta(x)$ for $\eta\in
L^2([u],\mu_{[u]})$.
Since $L^u$ and $L^v$ are unitarily equivalent so are $T^u$ and
$T^v$, and thus so are $N_u$ and $N_v$.
But now if $[u]\neq [v]$ then $[u]\cap[v]=\emptyset$ and
\cite[Lemma~4.15]{dana-ccr} implies that $N^u$ and $N^v$ are not
unitarily equivalent, a contradiction.  So $[u]= [v]$, and
hence $\Psi$ is injective.

To see that $\Psi$ is open onto its range we show that
$\Psi^{-1}:\range \Psi\to G^{(0)}/G$ is continuous. We argue by
contradiction: suppose that $[L^{u_n}]\to [L^{u}]$ in
$C^*(E;G,\lambda)$ but $[u_n]\not\to [u]$ in $G^{(0)}/G$. Let $U_0$
be an open neighbourhood of $[u]$ in $G^{(0)}/G$; let $q:G^{(0)}\to
G^{(0)}/G$ be the quotient map  and set $U=q^{-1}(U_0)$.  By passing
to a subsequence and relabeling we may assume that $u_n\notin U$ for
all $n$. By Lemma~\ref{lem-exact} $C^*(E|_U;G|_U)$ is  isomorphic to
an ideal $I$ of $C^*(E;G)$. Set $F:=G^{(0)}\setminus U$. Fix $f\in
C_c(E;G)$ such that $f(\gamma)=0$ for $\gamma\in E|_{F}$ (such $f$
are dense in $I$) and fix $\xi\in \H_{u_n}$.  Then
\begin{align*}
\|L^{u_n}(f)\xi\|^2_{u_n}&=\int_G|L^{u_n}(f)\xi(\gamma)|^2\, d\lambda_{u_n}(j(\gamma))\\
&=\int_G\Big( \int_G f(\gamma\alpha^{-1})\xi(\alpha)\,
d\lambda_{u_n}(j(\alpha)) \Big)^2\, d\lambda_{u_n}(j(\gamma)).
\end{align*}
Now consider the inner integrand: there
$r(\gamma\alpha^{-1})=r(\gamma)$ and $s(\gamma)=u_n$, so
$\gamma\alpha^{-1}\in E|_{[u_n]}$. But $U$ is saturated with
$u_n\notin U$, so $\gamma\alpha^{-1}\in E|_{F}$ and
$f(\gamma\alpha^{-1})=0$. Thus $\|L^{u_n}(f)\xi\|^2_{u_n}=0$, and
since $\xi$ was fixed we have $L^{u_n}(f)=0$.  It now follows that
$I\subseteq\ker L^{u_n}$ for all $n$.  But now $[L^{u_n}]\notin \hat
I$ for all $n$, contradicting that $\hat I$ is an open neighbourhood
of $[L^u]$ and $[L^{u_n}]\to [L^{u}]$.  Thus $\Psi$ is open.

(2) In view of (1) it suffices to show that $\Psi$ is surjective.
See the proof of \cite[Proposition~3.3]{MW92} (the proof given there
only uses that $\go/G$ is $T_1$).
\end{proof}

\begin{prop}\label{prop-orbits} Suppose that $E$ is a second-countable, locally compact, Hausdorff,
$\T$-groupoid such that $G:=E/\T$  is a principal groupoid with Haar system $\lambda$. Then
\begin{enumerate}
\item $C^*(E;G,\lambda)$ is liminal if and only if the orbit space $G^{(0)}/G$ is $T_1$; and
\item $C^*(E;G, \lambda)$ is postliminal if and only if the orbit space $G^{(0)}/G$ is $T_0$.
\end{enumerate}
\end{prop}
\begin{proof} Let $\sigma$  be the Haar system for $E$ from \eqref{eq-Haar}.

(1) Suppose $\cs(E;G,\lambda)$ is liminal. Since $C^*(E;G,\lambda)$
is separable, its spectrum is
 $T_1$  by \cite[9.5.3]{dix}.
By Proposition~\ref{MWthm}(1),  $\Psi: G^{(0)}/G \to
\cs(E;G,\lambda)^\wedge$, $[u]\mapsto [L^u]$ is a continuous
injection. So for each $u\in G^{(0)}$,
$\{[u]\}=\Psi^{-1}(\{[L^{u}]\})$ is closed in $G^{(0)}/G$. Thus
$G^{(0)}/G$ is $T_1$.

Conversely, suppose $G^{(0)}/G$ is $T_1$. Since all the isotropy
groups of $E$ are amenable, $\cs(E,\sigma)$ is liminal by
\cite[Theorem~6.1]{C2}. Now $\cs(E;G, \lambda)$ is liminal because
quotients of  liminal algebras are liminal.

(2) Proceed as in (1) using \cite[9.5.2]{dix} and
\cite[Theorem~6.1]{C2}.
\end{proof}

We now improve Proposition~\ref{MWthm} by using a composition series
to reduce to the $T_1$ case:

\begin{thm}
\label{MWthm-better} Suppose that $E$ is a second-countable, locally
compact, Hausdorff, $\T$-groupoid such that $G:=E/\T$  is a
principal groupoid with Haar system $\lambda$. If $G^{(0)}/G$ is
$T_0$  then $\Psi$ is a homeomorphism of $G^{(0)}/G$ onto
$C^*(E;G,\lambda)^\wedge$.
\end{thm}

\begin{proof}
We adapt the argument of \cite[Proposition~5.1]{C}. By
Proposition~\ref{MWthm}(1) it suffices to show that $\Psi$ is onto.
Since $G$ is $\sigma$-compact the equivalence relation
$R=\{(r(\gamma), s(\gamma)):\gamma\in G\}$ is an $F_\sigma$ set, so
$G^{(0)}/G$ is almost Hausdorff by \cite[Theorem~2.1]{ramsay}. A
Zorn's lemma argument (see the discussion on page~25 of
\cite{glimm}) gives an ordinal $\gamma$ and a collection
$\{V_\alpha:\alpha\leq \gamma\}$ of open subsets of $G^{(0)}/G$ such
that $V_0=\emptyset$, $V_\gamma= G^{(0)}/G$, $\beta<\alpha$ implies
$V_\beta\subseteq V_\alpha$, if $\alpha$ is a limit ordinal then
$V_\alpha=\cup_{\beta<\alpha} V_\beta$, and if $\alpha$ is not a
limit ordinal  or $0$ then $V_\alpha \setminus V_{\alpha-1}$ is an
open Hausdorff subset of $(G^{(0)}/G)\setminus V_{\alpha-1}$. Let
$q:G^{(0)}\to G^{(0)}/G$ be the  quotient map and set
$U_\alpha:=q^{-1}(V_\alpha)$.  By Lemma~\ref{lem-exact}, for each
$\alpha\leq \gamma$, there is an isomorphism $k_\alpha$ of
$C^*(E|_{U_\alpha};G|_{U_\alpha}, \lambda)$ onto  an ideal
$I_\alpha$  of $C^*(E; G, \lambda)$.

Now fix an irreducible representation $\pi$ of $C^*(E; G, \lambda)$.
Let $\alpha$ be the smallest element of the set $\{\lambda\leq
\gamma:\pi(I_\lambda)\neq 0\}$.  Note that $\alpha$ is not a limit
ordinal (if $\alpha$ were a limit ordinal then
$U_\alpha=\cup_{\beta<\alpha}U_\beta$ and $I_\alpha$ is the ideal
generated by the $I_\beta$ with $\beta<\alpha$.  But then
$\pi(I_\alpha)\neq 0$ implies $\pi(I_\beta)\neq 0$ for some
$\beta<\alpha$, contradicting the minimality of $\alpha$.) So
$\pi(I_\alpha)\neq 0$ and $\pi(I_{\alpha-1})=0$.  It follows that
$\pi$ is the canonical extension of the representation
$\pi|_{I_\alpha}$ and that $\pi|_{I_\alpha}$ factors through a
representation of $I_\alpha/I_{\alpha-1}$. By Lemma~\ref{lem-exact},
$I_\alpha/I_{\alpha-1}$ is isomorphic to $C^*(E|_{U_\alpha\setminus
U_{\alpha-1}};G|_{U_\alpha\setminus U_{\alpha-1}}, \lambda)$.   The
orbit space  $V_\alpha\setminus V_{\alpha-1}$ of
$G|_{U_\alpha\setminus U_{\alpha-1}}$  is Hausdorff, hence $T_1$.
 Now apply Proposition~\ref{MWthm}(2) to get that $\pi|_{I_\alpha}$ is unitarily
  equivalent $L^u(E|_{U_\alpha};G|_{U_\alpha})\circ k_\alpha^{-1}$ for some $u\in U_\alpha\setminus U_{\alpha-1}$.
Note that $L^u(E|_{U_\alpha};G|_{U_\alpha})\circ
k_\alpha^{-1}=L^u(E; G)|_{I_\alpha}$. Since  $\pi|_{I_\alpha}$ is
unitarily equivalent to  $L^u(E|_{U_\alpha};G|_{U_\alpha})\circ
k_\alpha^{-1}$ it follows that their canonical extensions are
unitarily equivalent. Thus $\pi$ is unitarily equivalent to $L^u(E;
G)$. So $\Psi$ is onto.
\end{proof}


\section{The twisted groupoid $\cs$-algebras with bounded trace}

We recall  \cite[Definition~3.1]{CaH}: a locally compact, Hausdorff
groupoid $G$ is \emph{integrable} if for every compact subset $N$ of
$G^{(0)}$,
\begin{equation}\label{eq-defn-integrable}
\sup_{x\in N}\{\lambda^x(s^{-1}( N))\}<\infty.
\end{equation}
Equivalently, by \cite[Lemma~3.5]{CaH},  $G$ is integrable if and
only if for each $z \in \go$, there exists an open neighbourhood $U$
of $z$ in $\go$ such that
\begin{equation}
\label{open} \sup_{x\in U}\{\lambda^x(s^{-1}(U))\}<\infty.
\end{equation}
%
So if a
groupoid fails to be integrable, then there exists a $z \in \go$ so
that
\begin{equation}
\label{open2} \sup_{x\in U}\{\lambda^x(s^{-1}(U))\}=\infty,
\end{equation}
for all open neighbourhoods $U$ of $z$, and we say $G$ \emph{fails to
be integrable at $z$.}

 In Proposition~\ref{prop1} we will prove that if $G$ is integrable, then $C^*(E;G,\lambda)$ has
 bounded trace.  To do so, we need to know that all irreducible representations of $C^*(E;G,\lambda)$ are equivalent to $L^u$ for some $u\in G^{(0)}$.
 We proved in \cite[Lemma~3.9]{CaH} that if $G$ is integrable, then all the orbits are closed;
unfortunately, there is a gap in the proof (the proof assumes
implicitly that orbits are locally closed at Equation 3.2 in
\cite{CaH}). Lemma~\ref{lem-locallyclosedorbits}  establishes that if $G$ is a principal
integrable groupoid, then the orbits are indeed locally closed. The
proof of \cite[Lemma~3.9]{CaH} then goes through as written. The proof of  Lemma~\ref{lem-locallyclosedorbits}  is based on the proof of \cite[Lemma~2.1]{AaH} which establishes  similar results in the transformation-group setting.

\begin{lemma}\label{lem-locallyclosedorbits}
Let $G$ be a second-countable, locally compact, Hausdorff, principal groupoid and let
$z\in G^{(0)}$.
\begin{enumerate}
\item If the orbit $[z]$ is not locally closed then for every open neighbourhood $V$ of $z$ in $G^{(0)}$, $\lambda_z(r^{-1}(V))=\infty$.
\item If $G$ is integrable then the orbits are locally closed.
\end{enumerate}
\end{lemma}

\begin{proof}
(1) Let $W$ be any open neighbourhood of $z$ in $G^{(0)}$.
We claim that for every compact neighbourhood $L$
of $z$ in $G$ there exists $\gamma_L\in G\setminus L$ such that
$s(\gamma_L)=z$ and $r(\gamma_L)\in W$.  Suppose there exists an $L$ for which no such $\gamma$ exists.
Since $[z]$ is not locally closed, 
$(\overline{[z]}\setminus[z])\cap W\neq\emptyset$. Let $y\in (\overline{[z]}\setminus[z])\cap W$. Then there exists
$\{\gamma_i\}\subseteq G$ such that $s(\gamma_i)=z$, $r(\gamma_i)\to
y$.  Then  $r(\gamma_i)\in W$ eventually.  So by assumption,
$\gamma_i\in L$ eventually. By passing to a subsequence  we may assume that $\gamma_i\to\gamma\in L$.    But now
$s(\gamma)=z$ and $r(\gamma)=y$, and hence $y\in[z]$, a
contradiction. This proves the claim.

Let $V$ be an open neighbourhood of $z$ in $G^{(0)}$ and let $M\in \P$.
There exists an open neighbourhood $U$ of $z$ in $G^{0}$
 and a compact symmetric neighbourhood $K$ of $z$ in $G$ such that $r(KU)\subseteq V$.
Let $c>0$ such that $\lambda_u(K)\geq c$ for $u\in U$ by [\textbf{6}, Lemma~3.10(1)]. 
Choose $k\in\P$ such that $kc>M$ and 
 $\gamma^{(1)},\dots,
\gamma^{(k)}$ as follows.  Set $\gamma^{(1)}=z$. By
the claim there exists $\gamma^{(2)}\in G\setminus
(K^2\gamma^{(1)})$ such that $s(\gamma^{(2)})=z$ and
$r(\gamma^{(2)})\in U$. Next, note that $K^2\gamma^{(1)}\cup K^2\gamma^{(2)}$ is compact, so by the claim  there exists
$
\gamma^{(3)}\in G\setminus (K^2\gamma^{(1)} \cup K^2\gamma^{(2)})
$
with $s(\gamma^{(3)})=z$ and $r(\gamma^{(3)})\in U$.  Continue.
 
Now
$r(\gamma^{(i)})\in U$ for $1\leq i\leq k$ and $\gamma^{(j)}(\gamma^{(i)})^{-1}\in G\setminus K^2$ when $i\neq j$.
Thus  $r(K\gamma^{(i)})\subseteq r(KU)\subseteq V$ and $K\gamma^{(i)}\cap K\gamma^{(j)}=\emptyset$
when $i\neq j$.
We have
\[
\lambda_z(r^{-1}(V))
\geq\lambda_z\Big(\bigcup_{i=1}^k K\gamma^{(i)}\Big)=\sum_{i=1}^k \lambda_z(K\gamma^{(i)})
=\sum_{i=1}^k \lambda_{r(\gamma^{(i)})}(K)\geq kc>M.
\]
Since $M$ was arbitrary, $\lambda_z(r^{-1}(V))=\infty$.

(2) Suppose there exists $z\in G^{(0)}$ such that $[z]$ is not
locally closed. Let $V$ be any open, relatively compact neighbourhood of$z$ in $G$.  By (1),
$\lambda_z(r^{-1}(V))=\infty$ thus
$\sup\{\lambda^x(s^{-1}(\overline{V})):x\in\overline{V}\}\geq \sup\{\lambda_x(r^{-1}(V)):x\in V\}=\infty$ 
so $G$ fails to integrable at $z$.
\end{proof}

\begin{prop}
\label{prop1} Suppose that $E$ is a second-countable, locally
compact, Hausdorff, $\T$-groupoid such that $G:=E/\T$  is a
principal groupoid with Haar system $\lambda$.  If $G$ is integrable
then the twisted groupoid $C^*$-algebra $\cs(E;G,\lambda)$ has
bounded trace.
\end{prop}

\begin{proof}
Since $G$ is an integrable principal groupoid, the orbits of $G$ are
locally closed by Lemma~\ref{lem-locallyclosedorbits}. By Theorem~\ref{MWthm-better}, $[u]\mapsto[L^u]$ is
a homeomorphism of $G^{(0)}/G$ onto the spectrum of $\cs(E;G,\lambda)$.  Fix $u\in \go$ and
$f\in C_c(E;G)^+$.    Then, by  \cite[Proposition~4.1]{MW92},  $L^u(f)$ is trace class with
\[
\tr(L^u(f)) = \int_G f(r(\gamma))\ d\lambda_u(j(\gamma)).
\]
Thus
\begin{align}
\tr(L^u(f))&\leq \|f\|_{\infty} \lambda_u \{ j(\gamma) : r(\gamma) \in \supp f\}\notag\\
&= \|f\|_{\infty} \lambda_u \{ j(\gamma) : \gamma \in r^{-1}(\supp f)\}\notag\\
&= \|f\|_{\infty} \lambda_u \{ j(\gamma) : j(\gamma) \in j(r^{-1}(\supp f))\}\notag\\
&= \|f\|_{\infty} \lambda_u \{ j(\gamma) : j(\gamma) \in r^{-1}(j(\supp f))\}\notag\\
&\leq \|f\|_{\infty} \underset{ u \in j(\supp
f)}{\sup}\{\lambda_u(r^{-1}(j(\supp f)))\}<\infty\notag
\end{align}
because $\supp f$ is compact and  $G$ is integrable.  Thus
$C_c(E;G)^+$ is contained in the ideal spanned by the bounded-trace
elements, and hence $C^*(E;G,\lambda)$ has bounded trace.
\end{proof}

In Theorem~\ref{thm-hard} below we show that if $\cs(E;G,\lambda)$
has bounded trace then $G$ is integrable. The proof is modeled
on \cite[Theorem~4.3]{MW92}, where Muhly and Williams prove that if
$\cs(E; G,\lambda)$ has continuous trace then $G$ is a proper
groupoid. Their proof strategy is the following.  Suppose that $G$
is not proper.  Then $G$ fails to be proper at some $z\in \go$. This
gives a sequence $\{x_n\}\subseteq G$ which is eventually disjoint
from every compact subset of $G$ and $r(x_n)$, $s(x_n)\to z$. (In
the terminology of \cite[Definition~3.6]{CaH}, $u_n:=s(x_n)$
converges $2$-times in $G^{(0)}/G$ to $z$.)

To show that $\cs(E;G,\lambda)$ does not have continuous trace, they
construct a function $F$ and show that $F^*F$ has certain tracial
properties. In particular, they partition the Hilbert space $\H_u$
of $L^u$ into  a direct sum $\H_{u,1}\oplus \H_{u,2}$ and write
$P_{u,i}$ for the projection of $\H_u$ onto $\H_{u,i}$. First they
show that $u\mapsto \tr(L^u(F^*F)P_{u,1})$ is continuous at $z$.
Second, they have a really clever and technical argument to show
that there is a constant $a>0$ such that
$\tr(L^{u_n}(F^*F)P_{u_n,2})\geq \|L^{u_n}(F^*F)P_{u_n,2}\|\geq a$
eventually. Next they push $F^*F$ into the Pedersen ideal of
\cite[Theorem~5.6.1]{ped} using a function $q\in C_c(0,\infty)$ such
that $q(t)=t$ for $t\in [a,\|F^*F\|]$. This gives an element
$d=q(F^*F)$ in the Pedersen ideal such that
$\tr(L^{u_n}(d)P_{u_n,2})\geq a>0$ eventually. It follows that
$[L^u]\mapsto \tr(L^u(d))$ is not continuous at $[L^z]$. Thus $d$ is
not a continuous-trace element.  But the Pedersen ideal is the
minimal dense ideal of a $C^*$-algebra, so the ideal spanned by the
continuous-trace elements cannot be dense in  $\cs(E;G,\lambda)$.
Thus $\cs(E;G,\lambda)$ does not have continuous trace.

\begin{thm}
\label{thm-hard}
 Suppose that $E$ is a second-countable, locally compact, Hausdorff, $\T$-groupoid such that $G:=E/\T$
 is a principal groupoid with Haar system $\lambda$. The following are equivalent:
 \begin{enumerate}
 \item the twisted groupoid $C^*$-algebra $\cs(E;G,\lambda)$ has bounded trace;
 \item $G$ is integrable; and
 \item $C^*(G)$ has bounded trace.
 \end{enumerate}
\end{thm}

Since $G$ is principal, (2) and (3) are equivalent by
\cite[Theorem~4.4]{CaH}. By Proposition~\ref{prop1}, if $G$ is
integrable then $\cs(E;G,\lambda)$ has bounded trace, so it
remains to show that (1) implies (2). We prove the contrapositive.
So suppose that $G$ is not integrable, say $G$ fails to be
integrable at  $z \in \go$. Then by \cite[Proposition~3.11]{CaH}
there exists a sequence $\{u_n\}$ in $\go$ so that $\{u_n\}$
converges to $z$, $u_n\neq z$ for all $n$, and $u_n$ \emph{converges
$k$-times in $\go/G$ to $z$}, for every $k \in \P$. That is, there
exist $k$ sequences $ \{ \gamma_n^{(1)}\},\
\{\gamma_n^{(2)}\},\dots,\{\gamma_n^{(k)}\}\subseteq G $ such that
\begin{enumerate}\label{ktimes}
\item $r(\gamma_n^{(i)})\to z$ and
$s(\gamma_n^{(i)})=u_n$ for $1\leq i\leq k$;
\item if $1\leq i<j\leq k$ then $\gamma_n^{(j)}(\gamma_n^{(i)})^{-1}\to\infty$ as $n\to\infty$, in the sense that $\{\gamma_n^{(j)}(\gamma_n^{(i)})^{-1}\}$
admits no convergent subsequence.
\end{enumerate}
We will prove that $\cs(E;G,\lambda)$ does not have  bounded trace.

Since $C^*$-algebras with bounded trace are liminal, we may assume
that the orbits are closed in $G^{(0)}$ by
Proposition~\ref{prop-orbits}. Let $M > 0$ be given.  In order to
show that $\cs(E;G,\lambda)$ does not have bounded trace, we will
show that there exists an element $d$ of the Pedersen ideal of
$\cs(E;G, \lambda)$ such that $\tr(L^{u_n}(d))> M$ eventually (see
\eqref{eq-Lu} for the definition of  the irreducible representations
$L^{u_n}$).  Since $M$ is arbitrary and the Pedersen ideal is the
minimal dense ideal \cite[Theorem~5.6.1]{ped}, this shows that the
ideal of bounded-trace elements cannot be dense.

We will use the same function $F$ as Muhly and Williams and adapt
their proof  as follows.
\begin{enumerate}
\item Show that there exists a constant $a>0$ such that $\|L^{u_n}(F^*F)P_{u_n,1}\|\geq a$.
\item Fix $l\in\N$ such that $la>M$.
Since $\{u_n\}$ converges $k$-times to $z$ in $\go/G$ for any $k$,
there exist $l$ sequences $\{u_n=\gamma_n^{(1)}\},
\{\gamma_n^{(2)}\}, ..., \{\gamma_n^{(l)}\}\subseteq G$ satisfying
the items (1) and (2) listed above.
\item Using that $\gamma_n^{(j)}(\gamma_n^{(i)})^{-1}\to\infty$, partition $\H_{u_n}$ into $l$ summands $\H_{u_n,i}$.
\item Write $P_{u_n,i}$ for the projection of $\H_{u_n}$ onto $\H_{u_n,i}$.  Show that, for every $1\leq i\leq l$,   $\|L^{u_n}(F^*F)P_{u_n, i}\|\geq a$ eventually.
\item Push $F^*F$ into the Pedersen ideal.
\end{enumerate}
Note that the order of events is subtle. We have to find the
constant $a$ before we can choose an appropriate $l$ and then get
$l$ sequences in $G$ which witness the $l$-times convergence of the
sequence $\{u_n\}$.    We retain, as much as possible, the notation
of \cite{MW92}.

We start by explaining the  function $F$ (see \eqref{eq-F} below).
Fix a function $g \in \cc^+(\go)$ so that $0 \leq g \leq 1$ and $g$
is identically one on a neighbourhood $U$ of $z$.  By
\cite[Lemma~2.7]{MW90} there exist symmetric, open, conditionally
compact neighbourhoods $W_0$ and $W_1$ in $G$ such that $ \go
\subseteq W_0 \subseteq \overline{W}_0 \subseteq W_1 $ and
$\overline{W_1}z\setminus W_0z\subseteq
r^{-1}(\go\setminus\supp(g))$. Choose symmetric, relatively compact,
open neighbourhoods $V_0$ and $V_1$ of $z$ in $G$ such that
$\overline{V_0}\subseteq V_1$.  Apply Lemma~\ref{danalemma1} to
obtain a compact neighbourhood $A$ of $z$ in $\go$ such that
\[
(\overline{W_1}^7 \overline{V_1}  \setminus W_0V_0)\cap G_A
 \subseteq r^{-1}(\go \setminus \supp g).\]
 Thus
\begin{equation}\label{eq-fixed}
(\overline{W_1}^7 \overline{V_1}\, \overline{W_1}^7 \setminus W_0V_0
W_0)\cap G_A \subseteq r^{-1}(\go \setminus \supp g).
\end{equation}
For $\gamma\in E_A$, set
\begin{equation}
g^{(1)}(\gamma) =
\begin{cases}
g(r(\gamma)), & \text{if $j(\gamma)\in\overline{W_1}^7 \overline{V_1}\,
\overline{W_1}^7$}\\
0, & \text{if $j(\gamma) \notin W_0V_0 W_0$},
\end{cases}
\notag
\end{equation}
where $j: E \to E/ \T$ is the quotient map.  This gives a well-defined function
$g^{(1)}$ on $E_A$ which, by Tietze's extension theorem extends to a well-defined function, also called   $g^{(1)}$, in $C_c(E)$.

Next choose a self-adjoint $b\in C_c(G)$ such that $0\leq b\leq 1$,
$b$ is identically one on $W_0V_0W_0^2V_0W_0$ and $b$ vanishes off
$\overline{W_1}^4\overline{V_1}\, \overline{W_1}^4$. Also choose a
compact neighbourhood $C$ of $z$ in $\go=E^{(0)}$ such that
$i(C\times\T)\supseteq \go\cap\supp g^{(1)}$. Define
$l:i(C\times\T)\to \T$ by $l(i(u,t))=1$ and extend $l$ to a function
$l\in C_c(E)$. By replacing $l$ by $(l+l^*)/2$ we may assume $l$ is
self-adjoint.  Set $h=\Upsilon(l)$ to obtain a self-adjoint $h\in
C_c(E;G)$ such that $h(i(u,1))=1$ for all $u\in C$. Finally,  define
\begin{equation}\label{eq-F}
F(\gamma) = g(r(\gamma))g(s(\gamma))b(j(\gamma))h(\gamma).
\end{equation}
Note that the $h$  ensures $F\in C_c(E;G)$, and that $F$ is self-adjoint because $h$ and $b$ are.

For each $n$, define
\[
\H_{u_n,1} = \H_{u_n}\cap L^2(E_{u_n} \cap j^{-1}(W_0V_0W_0),
\sigma_{u_n})
\]
and let $P_{u_n,1}$ be the projection onto $\H_{u_n,1}$. The
calculation \cite[pages 140--141]{MW92} shows that
$L^{u_n}(F)\H_{u_n,1}\subseteq \H_{u_n,1}$.

By \cite[Lemma~2.9]{MW90} there is a neighbourhood $V_2$ of $z$ in
$G$  and a conditionally compact, symmetric neighbourhood Y
of $\go$ in $G$ such that $V_2 \subseteq V_0$, and
\begin{equation}\label{eq-Y}
\gamma \in V_2\Longrightarrow r(Y\gamma)\subseteq U.
\end{equation}
 In particular, since $r(Y\gamma)=r(Yr(\gamma))$, we get $r(Y\gamma)\subseteq  U$ whenever
$r(\gamma) \in V_2$.  Since $u_n\to z$ we may assume that
$u_n\in V_2\cap C$ for all $n$.

The following lemma closely resembles \cite[Lemma~4.5]{MW92} and our
proof is similar; we replace the unbounded sequence $\{x_n\}$
appearing in \cite[Lemma~4.5]{MW92} with the sequence $\{u_n\}$
(corresponding to $\{r(x_n)\}$) which causes us to consider the
projection onto $\H_{u_n,1}$ rather than the projection onto
$\H_{u_n,1}^\perp$).

\begin{lemma}\label{lemDP}
\label{bound} Let $F$ be the function defined at \eqref{eq-F}. There
exist an $a >0$ and a neighbourhood $V_3\subseteq V_2\cap C$ of $z$
in $G$  such that
\[
\|L^{u_n}(F^*F)P_{u_n,1}\| \geq a
\]
whenever $u_n\in V_3$.
\end{lemma}

\begin{proof}
Let $Y$ be as above. Let  $\O_1,\O_2, c, Y_0$ be as in  \cite{MW92}.
Thus $\O_1$ and $\O_2$  are open neighbourhoods of $i(C \times 1)$ in
$E$ so that for every $\eta \in \O_1$, Re$(h(\eta))>\frac{1}{2}$ and
for every $\eta \in \O_2$, Re$(h(\eta))>\frac{1}{4}$, and $c$ is a
regular cross section of $j$ (see \cite[Proof of Lemma~3.2]{MW92}
for definition of regular). The set $Y_0$ is a conditionally
compact, symmetric  neighbourhood $Y_0$ of $\go$ in $G$ such that
$CY_0 \subseteq j(\O_1)$ and $Y\subseteq Y_0$.




Let $\T_0 = \{t_i\}_{i=1}^{\infty}$ be a countable dense subset
of $\T$.  If $x \in
Yu_n$, then
\[
j(c(u_n)c(x)^{-1}) = u_n x^{-1} \in CY\subseteq CY_0 \subseteq j(\O_1).\]
So there exists a $t \in \T_0$
so that $t \cdot c(u_n)c(x)^{-1} \in \O_1.$ Define $\zeta_n : Yu_n \to \T_0$ by $\zeta_n(y) = t_j$ where $j =
\min\{k:t_k \cdot c(u_n)c(y)^{-1} \in \O_1\}$.
Thus apart from its domain $\zeta$ is the function defined in  \cite[Lemma~4.5]{MW92};
that our $\zeta_n$ is Borel is proved as is done for the function in \cite[Lemma~4.6]{MW92}.

By another argument very similar to that of \cite[Lemma~2.9]{MW90}, there  exists a conditionally compact, symmetric neighbourhood $\tilde{Y}$ of $E^{(0)}$
in  $E$ such that $\tilde Yj^{-1}(CY)\subseteq \O_2$ and $j(\tilde Y)=Y$.
%
%
%
Since we are assuming that $u_n\in V_2\cap C$ for all $n$,  if $x\in Yu_n$ the claim gives
\begin{equation}\label{eq-O2}
\tilde{Y} \zeta_n(x)u_nc(x)^{-1} \subseteq \tilde YC\tilde Y\subseteq \tilde Yj^{-1}(CY)\subseteq \O_2
\end{equation}
for all $n$.


Muhly and Williams use a  function $t:E \to \T$ defined as follows:  for
each $\gamma \in E$, consider the element $\gamma
c(j(\gamma))^{-1}$.  This is in the image of $i$ and equals $i(u,s)$
for some $s \in \T$; then $t(\gamma) := s.$ Let $\chi_n$ be the characteristic function of $Yu_n$ and define
$\xi_n:E \to \C$ by
\[
\xi_n(\gamma) = t(\gamma)\zeta_n(j(\gamma))\chi_n(j(\gamma)).
\]
 Since all of the functions involved in defining $\xi_n$ are Borel,
so is $\xi_n$.
Also, it is clear that $\xi_n$ is bounded and has
compact support contained in $ \supp (\xi_n) \subseteq
j^{-1}(\overline{Yu_n})$.  Notice that  $t(s \cdot \gamma)
= st(\gamma)$ so that $\xi_n \in \H_{u_n}$; since also $Yu_n\subseteq W_0V_0W_0$ we have $\xi_n\in \H_{u_n,1}$.
(This is where we have departed from the Muhly-Williams proof - their unbounded sequence $\{x_n\}$ used
in place of our $\{u_n\}$ ensures their $\xi_n$ has support in the orthogonal complement of $\H_{u_n,1}$.)

Now fix $\gamma\in \tilde Yu_n$ and compute:\label{pg-integrals}
\begin{align}
L^{u_n}(F)(\xi_n)(\gamma) &= F*\xi_n(\gamma)\notag
= \int_G F(\gamma \alpha^{-1})\xi_n(\alpha) \
d\lambda_{u_n}(j(\alpha))\notag \\
 &= \int_G g(r(\gamma \alpha^{-1}))g(s(\gamma \alpha^{-1}))b(j(\gamma \alpha^{-1})) h(\gamma
\alpha^{-1})
\xi_n(\alpha) \ d\lambda_{u_n}(j(\alpha))\notag \\
 &=g(r(\gamma)) \int_G g(r(\alpha))b(j(\gamma \alpha^{-1})) h(\gamma
\alpha^{-1}) \xi_n(\alpha) \
d\lambda_{u_n}(j(\alpha)).\label{eq-666}
\end{align}
Note that the integrand is zero unless $j(\alpha)\in Yu_n$. Let
$j(\alpha)\in Yu_n$; then $j(\gamma\alpha^{-1}) \in Yu_nu_n^{-1}Y
\subseteq YV_0Y \subseteq W_0V_0W_0$, and hence $b(j(\gamma \alpha^{-1}) = 1$.
Also $j(\supp \xi_n ) \subseteq Yu_n$, so
\begin{align}
\eqref{eq-666}&= g(r(\gamma)) \int_{Yu_n} g(r(\alpha))h(\gamma \alpha^{-1})
\xi_n(\alpha) \ d\lambda_{u_n}(j(\alpha))\notag\\
&=g(r(\gamma)) \int_{Yu_n} g(r(x))h(\gamma c(x)^{-1}) \xi_n(c(x)) \
d\lambda_{u_n}(x)\notag
\intertext{by letting $x=j(\alpha)$ and
noting that $r(\alpha) = r(x)$ and $c(x)=c(j(\alpha))=\alpha$. Since
$r(\tilde Y u_n)=r(Yu_n)\subseteq U$ by our choice of $Y$ at \eqref{eq-Y} and since $g$ is identically one on
$U$, this is}
&= \int_{Yu_n} h(\gamma c(x)^{-1}) \xi_n(c(x)) \
d\lambda_{u_n}(x).\label{eq777}
\end{align}
By the definition of $\xi_n$ and using that $t\circ c=1$ we get $\xi_n(c(x))=\zeta_n(x)$ for $x\in Yu_n$, and since $h$ is $\T$-equivariant we get
\begin{align*}
\eqref{eq777}&= \int_{Yu_n} h(\gamma \zeta_n(x) c(x)^{-1}) \
d\lambda_{u_n}(x).
\end{align*}
But for $x\in Yu_n$, $\gamma \zeta_n(x) c(x)^{-1}\in\O_2$  by \eqref{eq-O2}, so $\text{Re}( h(\gamma \zeta_n(x) c(x)^{-1})) > \frac{1}{4}$. So
\[
\text{Re}\big(L^{u_n}(F)(\xi_n)(\gamma)\big) \geq \dfrac{1}{4} \lambda_{u_n}(Yu_n)
=\dfrac{1}{4} \lambda_{u_n}(Y)
\]
and hence
\[
|L^{u_n}(F)(\xi_n)(\gamma)|^2 \geq \frac{1}{16}\lambda_{u_n}(Y)^2.
\]
Now
\begin{align}
\|L^{u_n}(F)\xi_n \|^2
&= \int_G |L^{u_n}(F)(\xi_n)(\gamma)|^2 \ d\lambda_{u_n}(j(\gamma))\notag\\
&\geq\int_{\{j(\gamma): \gamma \in \tilde{Y}u_n\}}|L^{u_n}(F)(\xi_n)(\gamma)|^2 d\lambda_{u_n}(j(\gamma))\notag\\
&\geq \int_{Yu_n} \dfrac{1}{16} \lambda_{u_n}(Y)^2 \ d\lambda_{u_n}(j(\gamma))=\frac{1}{16}\lambda_{u_n}(Y)^3.\notag
\end{align}

By \cite[Lemma~3.10(2)]{CaH}, applied to the conditionally compact
neighbourhood $Y$,  there exists a neighbourhood $V_3$ of $z$ and
$k\in\N$ such that  if $v \in V_3$, $\lambda_v(Y) \geq k > 0$. By
shrinking, we may take $V_3\subseteq V_2\cap C$.

Let $u_n\in V_3$. Then $\lambda_{u_n}(Y) \geq k$. Now let  $a$ be a real number so that $0 <a \leq   \frac{1}{16}k^2$.
Since  $\|\xi_n\|^2 =  \lambda_{u_n}(Y)$  we get
\begin{align}
\|L^{u_n}(F^*F)P_{u_n,1}\|&= \|L^{u_n}(F)P_{u_n,1}\|^2 = \underset{\|\eta \| = 1}{\sup} \|L^{u_n}(F)\eta \|^2 \notag \\
 &\geq \|L^{u_n}(F)\dfrac{\xi_n}{\| \xi_n \|} \|^2=\frac{1}{16}\lambda_{u_n}(Y)^2\geq a\notag.\qedhere
 \end{align}
 \end{proof}

Now that we have our $a$, choose $l \in \P$ so that $la > M$.  Since $\{u_n\}$ converges $k$
times to $z$ in $\go/G$ for every $k$, there exist $l$ sequences
$\{u_n=\gamma_n^{(1)}\}, \{\gamma_n^{(2)}\}, ...,
\{\gamma_n^{(l)}\}$ satisfying the  two items on page~\pageref{ktimes}.

For each $n$ and $1\leq i \leq l$, we define the subspace
\[\H_{u_n,i} = \H_{u_n} \cap L^2(E_{u_n} \cap j^{-1}(W_0V_0W_0\gamma_n^{(i)}), \sigma_{u_n}).\]
Let $P_{u_n,i}$ be the projection onto $\H_{u_n,i}$ for $1 \leq i
\leq l.$

\begin{lemma}
The $\H_{u_n,i}\ (1\leq i\leq l)$ are invariant under $L^{u_n}(F)$,
and the $\H_{u_n,i}$ are eventually pairwise disjoint.
\end{lemma}

\begin{proof}
Fix $1\leq i\leq l$. To see that $\H_{u_n,i}$ is invariant under $L^{u_n}(F)$ it suffices to
show that $\supp(L^{u_n}(F)\psi)\subseteq j^{-1}(W_0V_0W_0 \gamma_n^{(i)})$ for $\psi\in \H_{u_n, i}$
with compact support. Fix $\gamma \in \supp(L^{u_n}(F)\psi)$.  Thus
\[
0\neq L^{u_n}(F)\psi (\gamma) = g(r(\gamma)) \int_G g(r(\alpha))
b(j(\gamma \alpha^{-1}))h(\gamma \alpha^{-1}) \psi(\alpha) \ d\lambda_{u_n} (j(\alpha)).
\]
For the integral to be non-zero, there must exist $\alpha\in s^{-1}(\{u_n\})$ in the support of the integrand.
Then $s(\gamma)=s(\alpha)=u_n$ and, in particular,  $j(\gamma)\in G_{u_n}\subseteq G_A$. Also $j(\alpha)\in W_0V_0W_0\gamma_n^{(i)}$.

Suppose, by way of contradiction, that $j(\gamma)\notin
\overline{W_1}^7\overline{V_1}\,\overline{W_1}^7\gamma_n^{(i)}$. We
have $j(\gamma\alpha^{-1})\in\supp b\subseteq
\overline{W_1}^4\overline{V_1}\,\overline{W_1}^4$. But now
$j(\gamma)=j(\gamma\alpha^{-1})j(\alpha)\subseteq
\overline{W_1}^4\overline{V_1}\, \overline{W_1}^7\gamma_n^{(i)}$,
contradicting that $j(\gamma)\notin
\overline{W_1}^7\overline{V_1}\,\overline{W_1}^7\gamma_n^{(i)}$. So
$j(\gamma)\in\overline{W_1}^7\overline{V_1}\,\overline{W_1}^7\gamma_n^{(i)}$.

Now suppose, again by way of contradiction, that $j(\gamma)\notin W_0V_0W_0\gamma_n^{(i)}$. Then
\[
r(\gamma)=r(j(\gamma))\in r\big(\big(
\overline{W_1}^7\overline{V_1}\,\overline{W_1}^7\gamma_n^{(i)}\setminus
W_0V_0W_0\gamma_n^{(i)} \big)\cap G_A  \big)\subseteq
r^{-1}(G^{(0)}\setminus\supp g)
\]
by \eqref{eq-fixed}. But now $g(r(\gamma))=0$, contradicting that
$L^{u_n}(F)\psi (\gamma)\neq 0$. Thus $j(\gamma)\in
W_0V_0W_0\gamma_n^{(i)}$.  Hence $\H_{u_n,i}$ is invariant under
$L^{u_n}(F)$ for  $1\leq i\leq l$.

Next, suppose that $1\leq i<j\leq l$ and that $\H_{u_n,j}$ and
$\H_{u_n,i}$ are not eventually disjoint. Then
$j^{-1}(W_0V_0W_0\gamma_n^{(j)})$ and
$j^{-1}(W_0V_0W_0\gamma_n^{(i)})$ are not eventually disjoint.
So there exists subsequences $\{\gamma_{n_k}^{(i)}\}$,
$\{\gamma_{n_k}^{(j)}\}$ of $\{\gamma_{n}^{(i)}\}$,
$\{\gamma_{n}^{(j)}\}$, and a sequence  $\{\alpha_k\}\subseteq  G$
such that
\[
\alpha_k\in W_0V_0W_0\gamma_{n_k}^{(i)}\cap  W_0V_0W_0\gamma_{n_k}^{(j)}.
\]
Thus $\alpha_k(\gamma_{n_k}^{(i)})^{-1}\in
W_0V_0W_0r(\gamma_{n_k}^{(i)})\subseteq W_0V_0W_0V_3$ and
$\alpha_k(\gamma_{n_k}^{(j)})^{-1}\in  W_0V_0W_0V_3$ eventually. So
\[
\gamma_{n_k}^{(j)}(\gamma_{n_k}^{(i)})^{-1}=\gamma_{n_k}^{(j)}s(\alpha_k)
(\gamma_{n_k}^{(i)})^{-1}=\gamma_{n_k}^{(j)}(\alpha_k)^{-1}\alpha_k(\gamma_{n_k}^{(i)})^{-1}\in
V_3^{-1}W_0V_0W_0^2V_0W_0V_3
\]
eventually.  But $V^{-1}_3W_0V_0W_0^2V_0W_0V_3$ is
relatively compact, so
$\{\gamma_{n_k}^{(j)}(\gamma_{n_k}^{(i)})^{-1}\}$ has a convergent
subsequence.  But this contradicts the $l$-times convergence of
$\{u_n\}$.
\end{proof}

\begin{lemma}
\label{bound2} Let $F$ be the function defined at  \eqref{eq-F} and
$a > 0$ be as in \lemref{bound}. Suppose  that $u_n\in V_3$ and
$r(\gamma_n^{(i)})\in V_3\cap C$ for $1\leq i\leq l$.  Then
\[
\|L^{u_n}(F^*F)P_{u_n, i}\| \geq a\] for $1 \leq i \leq l$.
\end{lemma}
Notice that in  \lemref{bound} above, we proved \lemref{bound2} in
the special case where $i=1$; we needed to do the base case $i=1$ to find the constant $a$. The proof of Lemma~\ref{bound2}  is similar to that of Lemma~\ref{bound}.

\begin{proof}
Let $Y, \O_1,\O_2, c, Y_0, \tilde Y, \T_0$ be as in
Lemma~\ref{lemDP}.  Fix $i$ and $x \in Y\gamma_n^{(i)}$. Then
\[
j(c(\gamma_n^{(i)})c(x)^{-1}) = \gamma_n^{(i)} x^{-1}\in
r(\gamma_n^{(i)})Y \subseteq CY\subseteq CY_0 \subseteq j(\O_1).\]
There exists a $t \in \T_0$ so that $t \cdot
c(\gamma_n^{(i)})c(x)^{-1} \in \O_1.$ Just as we defined $\zeta_n :
Yu_n \to \T_0$ in Lemma~\ref{lemDP} we now define
$\zeta_n^i:Y\gamma_n^{(i)}\to\T$ by $\zeta_n^i(y) = t_j$ where $j =
\min\{k:t_k \cdot c(\gamma_n^{(i)})c(y)^{-1}\in \O_1\}$. Since
$r(\gamma_n^{(i)})\in C$, we have
\[
\tilde Y\zeta_n^i(x)\gamma_n^{(i)}c(x^{-1})\subseteq\tilde Yr(\gamma_n^{(i)})\tilde Y\subseteq \tilde Yj^{-1}(CY)\subseteq\O_2.
\]
Let $\chi^i_n$ be the characteristic function of $Y\gamma_n^{(i)}$ and define
$\xi^i_n:E \to \C$ by
\[
\xi^i_n(\gamma) = t(\gamma)\zeta^i_n(j(\gamma))\chi^i_n(j(\gamma)).
\]
 Since all of the functions involved in defining $\xi^i_n$ are Borel,
so is $\xi^i_n$.
It is clear that $\xi^i_n$ is bounded, $\T$-invariant and has
compact support in $j^{-1}(\overline{Y\gamma_n^{(i)}})$. Since $s(\gamma_n^{(i)})=u_n$ we have
$\supp \xi_n^i\subseteq E_{u_n}\cap j^{-1}(Y\gamma_n^{(i)})\subseteq E_{u_n}\cap j^{-1}(W_0V_0W_0\gamma_n^{(i)}) $. Thus $\xi_n^i\in \H_{{u_n},i}$.

Fix $\gamma\in j^{-1}(Y\gamma_n^{(i)})$.  If $\alpha\in j^{-1}(Y\gamma_n^{(i)})$ then $j(\gamma\alpha^{-1})\in Yr(\gamma_n^{(i)})Y\subseteq W_0V_0W_0$
and hence $b(j(\gamma\alpha^{-1}))=1$.  The support of $\xi_n^i$ is $Y\gamma_n^{(i)}$, so the same calculation as done on page~\pageref{pg-integrals} gives
\begin{align*}
L^{u_n}(F)(\xi^i_n)(\gamma) &= g(r(\gamma)) \int_{Y\gamma_n^{(i)}} g(r(y))h(\gamma c(y)^{-1}) \xi^i_n(c(y)) \ d\lambda_{u_n}(y)
\intertext{which, since $r(\gamma_n^{(i)})\in V_2$  implies $r(Y\gamma_n^{(i)})\subseteq U$ and since $g$ is identically one on $U$, is}
&=\int_{Y\gamma_n^{(i)}} h(\gamma c(y)^{-1}) \xi^i_n(c(y)) \
d\lambda_{u_n}(y)
\\
&= \int_{Y\gamma_n^{(i)}} h(\gamma \zeta^i_n(y) c(y)^{-1}) \
d\lambda_{u_n}(y).
\end{align*}
But $\gamma \zeta^i_n(y) c(y)^{-1}\in\O_2$,  so $\text{Re}(
h(\gamma \zeta^i_n(y) c(y)^{-1})) > \frac{1}{4}$ and
\[
\text{Re}\big(L^{u_n}(F)(\xi^i_n)(\gamma)\big) \geq \dfrac{1}{4} \lambda_{u_n}(Y\gamma_n^{(i)})
=\dfrac{1}{4} \lambda_{r(\gamma_n^{(i)})}(Y).
\]
Since $r(\gamma_n^{(i)})\in V_3\cap C$ by assumption,  the same
calculation as at the end of the proof of Lemma~\ref{lemDP} gives
$\|L^{u_n}(F^*F)P_{u_n,i}\|\geq a$.
\end{proof}

To push $F^*F$ into the Pedersen ideal, let $q \in C_c(0,\infty)$ be
any function satisfying

 \begin{equation*}
q(t)=
\begin{cases} 0, & \text{if $t < \frac{a}{3}$,}
\\
2t - \dfrac{2a}{3}, &\text{if $ \frac{a}{3} \leq t < \frac{2a}{3}$,}
\\
t, &\text{if $\frac{2a}{3} \leq t \leq \|F^*F\|$.}
\end{cases}
\end{equation*}
Set $d := q(F^*F)$. We will show that $\tr(L^{u_n}(d)) \geq la>M$ eventually.

Fix $u_n\in V_3$ such that $r(\gamma_n^{(i)})\in V_3\cap C$ for $1\leq i\leq l$.
By  \lemref{bound2}, each $L^{u_n}(F^*F)P_{u_n,i}$ is
positive with an eigenvalue at least as large as $a$, and by choice of $q$ each
$q(L^{u_n}(F^*F)P_{u_n,i})$ is positive with norm at least as
large as $a$.

We claim that
\begin{equation*} \label{eq6} q(L^{u_n}(F^*F)P_{u_n,i})
= q(L^{u_n}(F^*F))P_{u_n,i}.
\end{equation*}
To see the claim, let $p$ be a polynomial that vanishes at $0$.
Since $L^{u_n}(F^*F)$ leaves $\H_{u_n,i}$ invariant, $L^{u_n}(F^*F)$
and $P_{u_n,i}$ commute. If we plug the operator
$L^{u_n}(F^*F)P_{u_n,i}$ into $p$ and simplify, we see that
$p(L^{u_n}(F^*F)P_{u_n,i}) = p(L^{u_n}(F^*F))P_{u_n,i}$. Because $q$
can be uniformly approximated by polynomials $p$, each vanishing at
0, the claim follows.
Now
\begin{align*}
\|L^{u_n}(d)\|=
\|(L^{u_n}(q(F^*F))\|
&\geq \sum_{i=1}^{l}
\|L^{u_n}(q(F^*F))P_{u_n,i}\|
= \sum_{i=1}^{l} \|q(L^{u_n}(F^*F))P_{u_n,i}\| \\
&=\sum_{i=1}^{l} \|q(L^{u_n}(F^*F)P_{u_n,i})\|\geq la>M
\end{align*}
by the choice of $l$.  Thus $\tr(L^{u_n}(d))>M$, and
since $M$ was arbitrary $d$ is not a bounded-trace element.
But $d$ is an element of the Pedersen ideal of
$\cs(E;G, \lambda)$,  so $\cs(E;G,\lambda )$ does not have bounded
trace. This completes the proof of Theorem~\ref{thm-hard}.


\section{The twisted groupoid $\cs$-algebras that are Fell Algebras}

Recall from \cite{MW90} that a groupoid $G$ is \emph{proper} if the
map $\pi:G\to \go\times\go$, defined by $\pi(\gamma)=(r(\gamma),
s(\gamma))$ for $\gamma\in G$, is a proper map. A subset $U$ of
$\go$ is \emph{wandering} if $G|_{U}=\pi^{-1}(U\times U)$
  is relatively compact.  Thus $G$ is proper if and only if every compact
  subset of $\go$ is wandering.  A groupoid $G$ where each unit has a wandering
  neighbourhood is called \emph{Cartan} \cite[Definition~7.3]{C}.

The following lemma illustrates the relationship between a Cartan
groupoid and 2-times convergence in the orbit space of the groupoid;
it is similar to one direction of \cite[Lemma~2.3]{AaH2}.
Lemma~\ref{lem-2times} will be used  in Example~\ref{ex-1} below.

\begin{lemma}\label{lem-2times}Let $G$ be a topological groupoid.
If there exists a sequence $\{u_n\}\subseteq \go$ which
converges $2$-times in $\go/G$ to $z\in\go$, then $G$ is not
Cartan.
\end{lemma}

\begin{proof}
We argue by contradiction. Suppose that $\{u_n\}\subseteq \go$
converges $2$-times in $\go/G$ to $z$ and that $G$ is Cartan.  Let
$U$ be a wandering neighbourhood of $z$ in $\go$, so that $G|_U$ is
relatively compact.   There exist sequences $ \{ \gamma_n^{(1)}\},\
\{\gamma_n^{(2)}\}\subseteq G $ such that $r(\gamma_n^{(i)})\to z$,
$s(\gamma_n^{(i)})=u_n$ for $i=1,2$ and
$\gamma_n^{(2)}(\gamma_n^{(1)})^{-1}\to\infty$ as $n\to\infty$. Thus
$\gamma_n^{(2)}, \gamma_n^{(1)}\in G|_U$ eventually, and hence
$\gamma_n^{(2)}(\gamma_n^{(1)})^{-1}\in \overline{G|_UG|_U}$,
eventually. But $\overline{G|_U G|_U}$ is compact, contradicting
that $\gamma_n^{(2)}(\gamma_n^{(1)})^{-1}\to\infty$ as $n\to\infty$.

\end{proof}

\begin{thm}
\label{thm-Fell}
Let $E$ be a second-countable, locally compact, Hausdorff, $\T$-groupoid such that $G:=E/\T$
 is a principal groupoid with Haar system $\lambda$. The following are equivalent:
 \begin{enumerate}
 \item the twisted groupoid $C^*$-algebra $\cs(E;G,\lambda)$ is a Fell algebra;
 \item $G$ is Cartan;
 \item $C^*(G)$ is a Fell algebra.
 \end{enumerate}
\end{thm}

\begin{proof} Since $G$ is principal, (2) and (3) are equivalent by
\cite[Theorem~7.9]{C}; we will now prove the equivalence of (1) and
(2).

Suppose  that $G$ is Cartan.  Fix an irreducible representation
$\rho$ of $\cs(E;G,\lambda)$;  we will show that $\rho$ satisfies
Fell's condition.  Since $G$ is Cartan, $\go/G$ is $T_1$ by
\cite[Lemma~7.4]{C}. So by Proposition~\ref{MWthm}, $\rho$ is
unitarily equivalent  to $L^u(E;G)$ for some $u\in\go$.  It suffices
to show  $L^u(E;G)$ satisfies Fell's condition.  Let $U_0$ be a
wandering neighbourhood of $u$ in $\go$ and $U=r(s^{-1}(U_0))$ its
saturation. Since $G$ has a Haar system, $r$ is open and hence $U$
is open. By \cite[Lemma~7.8]{C}, $G|_U$ is a proper groupoid, so by
\cite[Theorem~4.2]{MW90}, $C^*(E|_U; G|_U,\lambda)$ has continuous
trace.  By Lemma~\ref{lem-exact}, the inclusion $k:C_c(E|_U;
G|_U)\to C_c(E; G)$ induces an isometric isomorphism $k$ of
$C^*(E|_U; G|_U,\lambda)$ onto an ideal $I$ of $C^*(E; G,\lambda)$.
Thus $I$ has continuous trace and $L^u(E;G)|_I=L^u(E|_U;G|_u)\circ
k^{-1}$. Since $I$ has continuous trace, $L^u(E;G)|_I$ satisfies
Fell's condition in $\hat{I}$, and hence $L^u(E;G)$ satisfies Fell's
condition in $C^*(E; G,\lambda)^\wedge$. Thus $\rho$ satisfies
Fell's condition in $C^*(E; G,\lambda)^\wedge$ as well and $C^*(E;
G,\lambda)$ is a Fell algebra.

Conversely, suppose that $C^*(E; G,\lambda)$ is a Fell algebra.  Fix
$u\in\go$; we will show that $u$ has a wandering neighbourhood in
$\go$. Since $C^*(E; G,\lambda)$ is liminal, $\go/G$ is $T_1$ by
Proposition~\ref{prop-orbits} and $L:\go/G\to C^*(E;
G,\lambda)^\wedge$, $[u]\mapsto[L^u]$ is a homeomorphism by
Proposition~\ref{MWthm}.  By \cite[Corollary~3.4]{AS}, $[L^u]$ has
an open Hausdorff neighbourhood $O$ in $C^*(E; G,\lambda)^\wedge$.
Let $q:\go\to\go/G$ the quotient map and set $U=q^{-1}(L^{-1}(O))$.
Then $U$ is an open saturated subset of $\go$ and $C^*(E|_U;
G|_U,\lambda)$ is isomorphic to an  ideal $I$ of $C^*(E; G,\lambda)$
with spectrum $O$. Thus $C^*(E|_U; G|_U,\lambda)$ has continuous
trace (because $I$ has) and hence $G|_U$ is a proper groupoid by
\cite[Theorem~4.3]{MW90}.  So any relatively compact neighbourhood
contained in $U$ is a wandering neighbourhood of $u$ in $\go$. Thus
$G$ is Cartan.
\end{proof}


\section{Groupoids with abelian isotropy groups.}
Throughout this section  $\G$ is a second-countable, locally
compact, Hausdorff groupoid with Haar system $\lambda$. The
change in notation from $G$ to $\G$ is to emphasize that we are no
longer assuming that the groupoid $\G$ is principal. Let $A_u = \{
\gamma \in \G : r(\gamma) = s(\gamma) = u\}$ be the isotropy group
at $u\in \G^{(0)}$ and let $\cA = \{\gamma \in \G : r(\gamma) =
s(\gamma)\}$ be the isotropy groupoid; we also assume throughout
this section that the isotropy groups are abelian and vary
continuously, that is, that the map $u \mapsto A_u$ from $\G^{(0)}$
to the space of closed subsets of $\G^{(0)}$, is continuous in the
Fell topology. The isotropy groupoid acts on the left and right of
$\G$ and the quotient $\cR:=\G/\cA$ is a principal groupoid. The
main results of this section, Theorems~\ref{thm-end}
and~\ref{thm-Cartan2}, say that $C^*(\G,\lambda)$ has bounded trace
if and only if $\cR$ is an integrable groupoid, and that
$C^*(\G,\lambda)$ is a Fell algebra if and only if $\cR$ is a Cartan
groupoid.  Once again, our proofs are modeled after the analogous
result \cite[Theorem~1.1]{MRW96} for groupoid $C^*$-algebras with
continuous trace.

Since the isotropy groups vary continuously, $\cA$ has a Haar system
$\beta$ \cite[Lemmas~1.1 and 1.2]{renault-ideal}. Write $\hat\cA$
for the spectrum of $C^*(\cA,\beta)$. Then $\cR$ acts on the right
of $\hat \cA$ (see \eqref{eqhatA} and \eqref{eqRaction} below). In
\cite{MRW96} Muhly, Renault, and Williams show that if $\hat
\cA/\cR$ is Hausdorff, then  $\cs(\G, \lambda)$ is isomorphic to a
particular twisted groupoid $\cs$-algebra
\cite[Proposition~4.5]{MRW96}. They then apply the characterization
of when  twisted groupoid $C^*$-algebras have continuous trace from
\cite{MW92} to prove \cite[Theorem~1.1]{MRW96}.

Our strategy is similar.  We prove in Lemma~\ref{t1} that
$\G^{(0)}/\G$ is $T_1$ if and only  if $\hat{\cA}/\cR$ is $T_1$.
This allows us to show that the isomorphism of
\cite[Proposition~4.5]{MRW96} holds even if $\hat \cA/\cR$ is only
$T_1$.   Then we use the isomorphism and  our characterizations in
Theorems~\ref{thm-hard} and~\ref{thm-Fell} of when  twisted groupoid
$C^*$-algebras  have bounded trace or are Fell algebras to get
results for  $C^*(\G,\lambda)$.

We need some background before we can proceed to Lemma~\ref{t1}.
Since $C^*(\cA,\beta)$ is a separable commutative $C^*$-algebra, the
discussion on \cite[p.~3630]{MRW96} shows that
\begin{equation}\label{eqhatA}
\hat \cA=\{(\chi, u):u\in \G^{(0)}, \chi\in\hat A_u\}
\end{equation}
where $(\chi, u)(f)=\int_{A_u}\chi(a)f(a)\, d\beta^u(a)$ for $f\in
C_c(\cA)$. Proposition~3.3 of \cite{MRW96}  describes criteria for
convergence in $\hat \cA$: $(\chi_n, u_n)\to (\chi, u)$ in $\hat
\cA$ if and only if (1) $u_n\to u$ in $\cA^{(0)} (=\G^{(0)})
$, and (2) if $a_n\in A_{u_n}$,  $a\in A_u$  and $a_n\to a$ in
$\cA$, then $\chi_n(a_n)\to\chi(a)$.

If $\chi\in\hat A_u$ and $\gamma\in \G$ with $r(\gamma)=u$, then
$\chi\cdot\gamma$ is the character  of $A_{s(\gamma)}$ defined by
$\chi\cdot\gamma(a)=\chi(\gamma^{-1}a\gamma)$. Note that
$\chi\cdot\gamma$ depends only on $\dot\gamma$.  There is a groupoid
action  of $\cR$ (and $\G$) on the
right of $\hat{\cA}$ via
\begin{equation}\label{eqRaction}
(\chi, u)\cdot \dot{\gamma} = (\chi \cdot \gamma, s(\gamma))
\end{equation}
for $\gamma\in\G$ with $r(\gamma)=u$.

\begin{lemma}
\label{t1} Suppose that $\G$ is a second-countable,  locally
compact, Hausdorff groupoid with  Haar system. Also assume that the
isotropy groups are abelian and vary continuously. Then
$\G^{(0)}/\G$ is $T_1$ if and only if $\hat{\cA}/\cR$ is $T_1$.
\end{lemma}

\begin{proof} First suppose that $\G^{(0)}/\G$ is $T_1$.
Fix $(\rho, v)\in\Hat\cA$. It suffices to show that $[(\rho, v)]$ is
closed.  Let $(\chi_n,u_n)\in [(\rho, v)]$ and suppose that
$(\chi_n, u_n)\to (\chi, u)$ in $\hat \cA$.  Thus there exists
$\gamma\in \G$ with $s(\gamma)=u$ and $r(\gamma)=v$, and, for each
$n$, there exists $\gamma_n\in\G$ with $s(\gamma_n)=u_n$,
$r(\gamma_n)=v$ such that
$(\chi_n,u_n)=(\rho\cdot\gamma_n,u_n)=(\rho,v)\cdot\dot{\gamma_n}$.
Note  $u\in [v]$ since $u_n\in [v]$ and $[v]$ is closed by
assumption, and that $\gamma,\gamma_n\in \G|_{[v]}$.

Since $\G$ has a Haar system, $r$ and $s$ are open maps
\cite[Proposition~2.4]{renault} and this puts us in the setting of
\cite{MRW96}. Since $\G|_{[v]}$ is a transitive groupoid, the map
$\pi:\G|_{[v]}\to [v]\times[v]$,
$\pi(\alpha)=(r(\alpha),s(\alpha))$ is open by \cite[Theorems~2.2A
and 2.2B]{MRW96}.  Since $\pi(\gamma_n)=(v, u_n)\to
(v,u)=\pi(\gamma)$ and $\pi$ is open, there exists a subsequence
$\{\gamma_{n_k}\}$ and a sequence $\{\eta_k\}\subseteq\G$ such that
$\pi(\gamma_{n_k})=\pi(\eta_k)$ and $\eta_k\to\gamma$ in $\G|_{[v]}$
(see, for example, \cite[Proposition~1.15]{tfb2}). Thus
$\eta_k\to\gamma$ in $\G$ as well. Note that
$\dot\gamma_{n_k}=\dot\eta_k$.

Fix a sequence $\{a_k\}$ with $a_k\in A_{u_{n_k}}$ such that $a_k\to
a$ in $\hat\cA$. Then by the continuity  of multiplication,
\[
\chi_k(a_k)=(\rho\cdot\dot\gamma_{n_k})(a_k)=\rho(\eta_k^{-1}a_k\eta_k)\to\rho(\gamma^{-1}a\gamma)=(\rho\cdot\dot\gamma)(a).
\]
Thus $\{(\chi_{n_k},u_{n_k})\}$ converges to both $(\chi,u)$ and $(\rho\cdot\dot\gamma,u)$ in $\hat\cA$.  Since $\hat\cA$ is Hausdorff we have
\[
(\chi,u)=(\rho\cdot\dot\gamma, u)=(\rho,v)\cdot\dot\gamma\in[(\rho,v)].
\]
So $[(\rho,v)]$ is closed.  Hence $\hat\cA/\cR$ is $T_1$.

For the converse, first consider $\phi:\G^{(0)}\to\hat\cA/\cR$
defined by $\phi(u)=[(1_u,u)]$, where $1_u$ is the  trivial
character $a\mapsto 1$ for $a\in A_u$. If $u_n\to u$ in $\G^{(0)}$
then, using the convergence criteria for sequences in $\hat\cA$ of
\cite[Proposition~3.3]{MRW96}, it is clear that  $(1_{u_n},u_n)\to
(1_u,u)$ in $\hat \cA$, so that $\phi$ is continuous. Now suppose
that $\phi(u)=\phi(v)$. Then there exists $\gamma\in\G$ with
$r(\gamma)=u$ such that $(1_u,u)\cdot\dot\gamma=(1_v,v)$.  Thus
$v=s(\gamma)$ and hence $[u]=[v]$.  So $\phi$ induces a continuous
injection $\phi:\G^{(0)}/\G\to\hat\cA/\cR$. It follows that
$\G^{(0)}/\G$ is $T_1$ if $\hat\cA/\cR$ is.
\end{proof}

In \cite{MRW96}, Muhly, Renault, and Williams define a groupoid
$\hat \cA\rtimes\cR$ as follows. As a set $\hat
\cA\rtimes\cR=\{(\chi,r(\gamma),\dot\gamma)\in\hat\cA\times\cR\}$,
but an element $(\chi,r(\gamma),\dot\gamma)$ is  abbreviated to just
$(\chi,\dot\gamma)$; the topology on  $\hat \cA\rtimes\cR$ is the
product topology. The unit space is $\hat \cA$ with  range and
source maps
\[
r((\chi,\dot\gamma))=(\chi, r(\gamma))\text{\ and\ }s((\chi,\dot\gamma))=(\chi\cdot\gamma, s(\gamma)).
\]
The multiplication and inverse in $\hat \cA\rtimes\cR$ is given by
\[
(\chi,\dot\gamma)(\chi\cdot\dot\gamma,\dot\alpha)=(\chi,\dot\gamma\dot\alpha)\text{\ and\ }(\chi,\dot\gamma)^{-1}=(\chi\cdot\dot\gamma,\dot\gamma^{-1}).
\]
It is straightforward to see that $\hat \cA\rtimes\cR$ is principal.
Note that  $\cR$ is proper if and only if
$\hat{\cA} \rtimes \cR$ is proper; similarly one is Cartan or integrable if and only if the other is:

\begin{lemma}
\label{intR} Suppose that $\G$ is a second-countable,  locally
compact, Hausdorff groupoid with abelian isotropy.  Also assume that
the isotropy groupoid $\cA$ has a Haar system.
\begin{enumerate}
\item If $\G$ has a  Haar system, then $\cR$ and  $\hat{\cA} \rtimes \cR$ have Haar systems  $\alpha$ and $\delta\times\alpha$, respectively;
and with respect to these Haar systems, $\cR$ is integrable if and only if $\hat{\cA} \rtimes \cR$ is integrable.
\item $\cR$ is Cartan if and only if $\hat{\cA} \rtimes \cR$ is Cartan.
\end{enumerate}
\end{lemma}

\begin{proof} (1) Since $\G$ and $\cA$ have Haar systems, $\cR$ has a Haar system $\alpha$ by \cite[Lemma~4.2]{MRW96}.
It is straightforward to check that if
$(\chi,u)\in\hat\cA$ and $\delta_{(\chi,u)}$ is point-mass measure,
then
$\delta\times\alpha^{(\chi,u)}:=\delta_{(\chi,u)}\times\alpha^u$
gives a Haar system on $\hat\cA\rtimes\cR$.

Suppose $\cR$ is integrable.  Fix a  compact subset  $K$  in
$(\hat{\cA} \rtimes \cR)^{(0)} = \hat{\cA}$.   Let
$p_2:\hat\cA\rtimes\cR\to\cR$ be the projection onto the second
coordinate; note that $p_2(K)$ is a compact subset of
$R^{(0)}=\G^{(0)}$.  We have
\begin{align*}
(\delta\times\alpha)^{(\chi,u)}(s^{-1}(K))
&=\delta_{(\chi,u)}\times\alpha^u\big(\{(\eta,\dot\gamma)\in\hat\cA\rtimes\cR:(\eta\cdot\gamma,s(\gamma))\in K \} \big)\\
&=\alpha^u\big(\{\dot\gamma\in\cR: r(\gamma)=u, (\chi\cdot\gamma,s(\gamma))\in K \} \big)\\
&\leq \alpha^u\big(\{\dot\gamma\in\cR: s(\gamma)\in p_2(K)\}\big)\\
&=\alpha^u\big(s^{-1}(p_2(K))\big).
\end{align*}
Since $\cR$ is integrable, this gives
\[
\sup_{(\chi,u)\in K}\big\{(\delta\times\alpha)^{(\chi,u)}(s^{-1}(K))\big\}\leq
\sup_{u\in p_2(K)}\big\{\alpha^u(s^{-1}(p_2(K))\})\big\}<\infty.
\]
So $\hat\cA\rtimes\cR$ is integrable.

Conversely, suppose that $\hat{\cA} \rtimes \cR$ is integrable. Fix
a compact set $L$ in $\cR^{(0)} = \G^{(0)}$.  For each $u\in L$, let
$1_u$ be the trivial character of $A_u$, so that $a\mapsto 1$ for
all $a\in A_u$.  Set $\tilde{L} = \{(1_u, u) : u \in L\} \subseteq
\hat{\cA}$. We claim that $\tilde{L}$ is a compact subset of $(\hat
\cA\rtimes\cR)^{(0)}=\hat \cA$.   To see this, let $\{(1_{v_n},
v_n)\}$ be a sequence in $\tilde L$. Then $\{v_n\}$ is a sequence in
$L$ and hence has a convergent subsequence $v_{n_k}\to v\in L$.
Using the convergence criteria for sequences in $\hat \cA$ of
\cite[Proposition~3.3]{MRW96} it is clear that
$(1_{v_{n_k}},v_{n_k})\to (1_v, v)$ in $\hat\cA$. Thus $\tilde{L}$
is compact in $\hat{\cA}$.

Now note that $s((1_u,\dot\gamma))=(1_u, s(\gamma))$, so that $s((1_u,\dot\gamma))\in\tilde L$ if and only if $s(\dot\gamma)\in L$. Thus
\begin{align*}
\sup_{u\in L}\big\{\alpha^u(s^{-1}( L))\big\}
\leq \sup_{(1_u,u)\in
\tilde{L}}\{(\delta \times \alpha)^{(1_u,u)}(s^{-1}( \tilde{L}))\}<
\infty
\end{align*}
because $\hat{\cA} \rtimes \cR$  is integrable and $\tilde L$ is a compact subset of its unit space. So $\cR$ is integrable.

(2)  First suppose that $\cR$ is Cartan.  Fix $(\chi,
u)\in\hat\cA=(\hat \cA\rtimes \cR)^{(0)}$. Let $K$ be a relatively
compact, wandering neighbourhood of $u$ in $\cR^{(0)}$. Let
$p_1:\hat\cA\rtimes\cR\to\cA$,  $p_2:\hat\cA\rtimes\cR\to\cR$ be the
projections onto the first and second coordinate, respectively. Let
$N$ be a relatively compact neighbourhood of $(\chi, u)$ in $\hat\cA$
such that $p_2(N)=K$.  Let $\{(\eta_n,\dot\gamma_n)\}$ be a sequence
in $\pi^{-1}(N\times N)=\{(\chi,\dot\gamma):(\chi,r(\gamma))\in N,
(\chi\cdot\dot\gamma,s(\gamma))\in N\}$.  Then
$\{\dot\gamma_n\}\subseteq\pi^{-1}(p_2(N)\times
p_2(N))=\pi^{-1}(K\times K)$, hence has a convergent subsequence
$\{\dot\gamma_{n_k}\}$. Note $\{\eta_{n_k}\}\subseteq p_1(N)$, a
relatively compact set. So there exists a convergent subsequence
$\{\eta_{n_{k_i}}\}$.  So $\{(\eta_{n_{k_i}}\dot\gamma_{n_{k_i}}\}$
is a convergent subsequence of $\{(\eta_n,\dot\gamma_n)\}$.  Thus
$\pi^{-1}(N\times N)$ is relatively compact. Hence $\hat \cA\rtimes
\cR$ is Cartan.

Conversely, suppose that $\hat \cA\rtimes \cR$ is Cartan.  Fix $u\in
\cR^{(0)}$.  There exists  a wandering neighbourhood $N$ of $(1_{u},
u)$ in $\hat\cA$.  Let $K=p_2(N)$; then $K$ is a neighbourhood of
$u$.  Let $\{\dot\gamma_n\}\subseteq\pi^{-1}(K\times K)$.  For each
$n$ there exists $\eta_n$ such that
$(\eta_n,\dot\gamma_n)\in\pi^{-1}(N\times N)$.  But
$\pi^{-1}(N\times N)$ is relatively compact, so
$\{(\eta_n,\dot\gamma_n)\}$ has a convergent subsequence
$\{(\eta_{n_k},\dot\gamma_{n_k})\}$. Thus $\{\dot\gamma_{n_k}\}$ is
a convergent subsequence of $\{\dot\gamma_n\}$. Thus
$\pi^{-1}(K\times K)$ is relatively compact.  Hence $\cR$ is Cartan.
\end{proof}

We will now briefly describe the $\T$-groupoid $\cD$  of \cite[\S4]{MRW96}.  There
\begin{equation}\label{eq-D}
\cD:=\{ (\chi,z,\gamma):\chi\in\hat A_{r(\gamma)}, z\in\T, \gamma\in\G \}/\!\sim,
\end{equation}
where $(\chi, \chi(a)z,\gamma)\sim(\chi, z, a\cdot\gamma)$; the unit space is $\hat\cA$ with
\[
r([(\chi,z,\gamma)])=(\chi, r(\gamma))\quad\text{and}\quad s([(\chi,z,\gamma)])=(\chi\cdot\gamma, s(\gamma))
\]
and multiplication and inverse
\[[(\chi,z,\gamma)][(\chi\cdot\gamma,z',\gamma')]=[(\chi,zz',\gamma\gamma')]\quad
\text{and}\quad[(\chi,z,\gamma)] ^{-1}=[(\chi\cdot\gamma,\bar
z,\gamma^{-1})].\] That $\cD$ is indeed a $\T$-groupoid over
$\hat\cA\rtimes\cR$ is established on \cite[p.~3636]{MRW96}.

Proposition~4.5 of \cite{MRW96} says that if $\hat \cA/\cR$
is Hausdorff,  then $\cs(\G)$ and $\cs(\cD;\hat{\cA} \rtimes \cR,
\delta \times \alpha)$ are isomorphic.  We now establish that the
given proof works almost as is written even if $\hat \cA/\cR$ is
only $T_1$.  Proposition~4.5 of \cite{MRW96}  uses the Hausdorff
assumption in three places.  The first use is in Lemma~4.8 to
establish that the $\G$-orbits in $\G^{(0)}$ are closed; so assuming
$\hat \cA/\cR$ is  $T_1$ suffices by  Lemma~\ref{t1}. The second use
is to establish again that the $\G$-orbits are closed in $\G^{(0)}$
so that \cite[Lemma~2.11]{MRW96} applies. The third use is to
establish that every irreducible representation of
$\cs(\cD;\hat{\cA} \rtimes \cR, \delta \times \alpha)$ is of the
form $[L^{(\chi,u)}]$; here we note that $(\hat{\cA} \rtimes
\cR)^{(0)}/(\hat{\cA} \rtimes \cR)$ is homeomorphic to $\hat
\cA/\cR$, so we  can use Proposition~\ref{MWthm} for this if $\hat
\cA/\cR$ is  $T_1$. Thus we have:

\begin{prop}
\label{isomo}
 Suppose $\G$ is a second-countable, locally compact, Hausdorff groupoid
with Haar system $\lambda$.  Also suppose that the isotropy groups
of $\G$ are abelian and vary continuously. If $\G^{(0)}/\G$ is $T_1$,
then $\cs(\G,\lambda)$ and $\cs(\cD;\hat{\cA} \rtimes \cR, \delta
\times \alpha)$ are isomorphic.
\end{prop}

\begin{thm}\label{thm-end}
Suppose that $\G$ is a second-countable,  locally compact, Hausdorff
groupoid with Haar system $\lambda$. Also suppose that the isotropy
groups of $\G$ are  abelian and vary continuously. Let $\cA$ be the
isotropy groupoid. The following are equivalent:
\begin{enumerate}
\item $\cs(\G,\lambda)$ has bounded trace;
\item $\cR:=\G/\cA$ is integrable;
\item $C^*(\cR)$ has bounded trace.
\end{enumerate}
\end{thm}

\begin{proof} Since $\cR$ is principal, the equivalence of (2) and (3) is \cite[Theorem~4.4]{CaH}; we will now prove the
equivalence of (1) and (2). Note that the isotropy groups vary continuously if and only if
the isotropy groupoid $\cA$ has a Haar system by \cite[Lemmas~1.1 and 1.2]{renault-ideal}.

First suppose that $\cs(\G,\lambda)$ has bounded trace.  Then
$\cs(\G,\lambda)$ is liminal and hence the orbits of $\G$ are closed
by \cite[Theorem~6.1]{C2}.  Now $\cs(\G, \lambda)$ and
$\cs(\mathcal{D};\hat{\cA} \rtimes \cR, \delta \times \alpha)$  are
isomorphic by Proposition~\ref{isomo}.  Thus
$\cs(\mathcal{D};\hat{\cA} \rtimes \cR, \delta \times \alpha)$ has
bounded trace as well.  Thus $\hat{\cA} \rtimes \cR$ is integrable
by Theorem~\ref{thm-hard}, and hence $\cR$ is integrable by
Lemma~\ref{intR}(1).

Conversely,  suppose $\cR$ is integrable.  Then the orbits in $\cR$
are closed  by \cite[Lemma~3.9]{CaH} and
Lemma~\ref{lem-locallyclosedorbits}. Since $\cR$ and $\G$ have the
same orbit space, orbits are closed in $\G$.   By
Proposition~\ref{isomo}, $\cs(G,\lambda)$ and
$\cs(\mathcal{D};\hat{\cA} \rtimes \cR, \alpha)$  are isomorphic.
Since $\cR$ is integrable,  $\hat{\cA} \rtimes \cR$ is integrable by
Lemma~\ref{intR}(1).  Thus $\cs(\mathcal{D};\hat{\cA} \rtimes \cR,
\alpha)$, and hence  $\cs(\G,\lambda)$, has bounded trace by
Theorem~\ref{thm-hard}.
\end{proof}

\begin{thm}\label{thm-Cartan2}
Suppose that $\G$ is a second-countable,  locally compact, Hausdorff
groupoid with Haar system $\lambda$. Also suppose that the isotropy
groups of $\G$ are  abelian and vary continuously. Let $\cA$ be the
isotropy groupoid. The following are equivalent:
\begin{enumerate}
\item $\cs(\G,\lambda)$ is a Fell algebra;
\item $\cR:=\G/\cA$ is Cartan;
\item $C^*(\cR)$ is a Fell algebra.
\end{enumerate}
\end{thm}

\begin{proof} Since $\cR$ is principal, the equivalence of (2) and (3) is \cite[Theorem~7.9]{C}.
The proof of the equivalence of (1) and (2) is  similar to the proof of Theorem~\ref{thm-end}, using Lemma~\ref{intR}(2), Theorem~\ref{thm-Fell}
and \cite[Lemma~4.7]{C} in place of Lemma~\ref{intR}(1),Theorem~\ref{thm-hard} and \cite[Lemma~3.9]{CaH}, respectively.
\end{proof}



\section{Examples} Our examples use  groupoids constructed from directed graphs.  We start with some background.
Let  $E=(E^0,E^1,r,s)$  be a directed graph. Thus $E^0$ and $E^1$
are countable sets of vertices and edges,  respectively, and
$r,s:E^1\to E^0$ are the range and source map, respectively. For
$e\in E^1$, call $s(e)$ the  source of $e$ and $r(e)$ the range of
$e$. A directed graph $E$ is row-finite if $r^{-1}(v)$ is finite for
every $v\in E^0$. A finite path is a finite sequence
$\alpha=\alpha_1\alpha_2\cdots\alpha_k$ of edges $\alpha_i\in E^1$
with $s(\alpha_j)=r(\alpha_{j+1})$ for $1\le j\le k-1$; write
$s(\alpha)=s(\alpha_k)$ and $r(\alpha)=r(\alpha_1)$, and call
$|\alpha|:=k$ the  length of $\alpha$. An  infinite path
$x=x_1x_2\cdots$ is defined similarly, although $s(x)$ remains
undefined. Let $E^*$ and $E^\infty$ denote the set of all finite
paths and infinite paths in $E$ respectively. If
$\alpha=\alpha_1\cdots\alpha_k$ and $\beta=\beta_1\cdots\beta_j$ are
finite paths with $s(\alpha)=r(\beta)$, then $\alpha\beta$ is the
path $\alpha_1\cdots\alpha_k\beta_1\cdots\beta_j$. When $x\in
E^\infty$ with $s(\alpha)=r(x)$ define $\alpha x$ similarly. A cycle
is a finite path $\alpha$ of non-zero length such that
$s(\alpha)=r(\alpha)$.  By \cite[Corollary~2.2]{KPRR}, the cylinder
sets
\[
Z(\alpha):=\{x\in E^\infty
:x_1=\alpha_1,\ldots,x_{|\alpha|}=\alpha_{|\alpha|}\},
\]
parameterized by $\alpha\in E^*$, form a basis of compact, open sets
for a locally compact, $\sigma$-compact,  totally disconnected,
Hausdorff topology on $E^\infty$.

In  \cite{KPRR}, Kumjian, Pask, Raeburn and Renault built a groupoid
$\G_E$, called the path groupoid,  from a row-finite directed
graph $E$ as follows. Two paths $x,y\in E^\infty$ are shift
equivalent with lag $k\in\Z$ (written $x\sim_k y$) if there exists
$N\in\N$ such that $x_i=y_{i+k}$ for all $i\ge N$. Then  the
groupoid is
\[ \G_E:=\{(x,k,y)\in E^\infty\times \Z\times
E^\infty :x\sim_k y\}.
\]
with composable pairs
\[
\G_E^{(2)}:=\{\big((x,k,y),(y,l,z)\big):(x,k,y),(y,l,z)\in \G_E\},
\]
and composition and inverse given by
\[
(x,k,y)\cdot (y,l,z):=(x,k+l,z)\quad\text{and}\quad(x,k,y)^{-1}:=(y,-k,x).
\]

For each $\alpha,\beta\in E^*$ with $s(\alpha)=s(\beta)$, let $Z(\alpha,\beta)$ be the set
\[
\{(x,k,y):x\in Z(\alpha), y\in Z(\beta), k=|\beta|-|\alpha|, x_i=y_{i+k}\text{ for } i>|\alpha|\}.
\]
By \cite[Proposition~2.6]{KPRR}, the collection of sets
\[
\{Z(\alpha,\beta):\alpha,\beta\in E^*, s(\alpha)=s(\beta)\}
\]
is a basis of compact, open sets for a second-countable, locally
compact, Hausdorff topology on $\G_E$ such that  $\G_E$ is an
r-discrete groupoid with a Haar system of counting measures.  After
identifying each $(x,0,x)\in \G_E^{(0)}$ with $x\in E^\infty $,
\cite[Proposition~2.6]{KPRR} says that the topology on $\G_E^{(0)}$
is identical to the topology on $E^\infty $. We caution that in our
notation (which is now standard) the sources and ranges are swapped
from the notation used in \cite{KPRR}.

\begin{example}\label{ex-1} Let $E$ be the graph
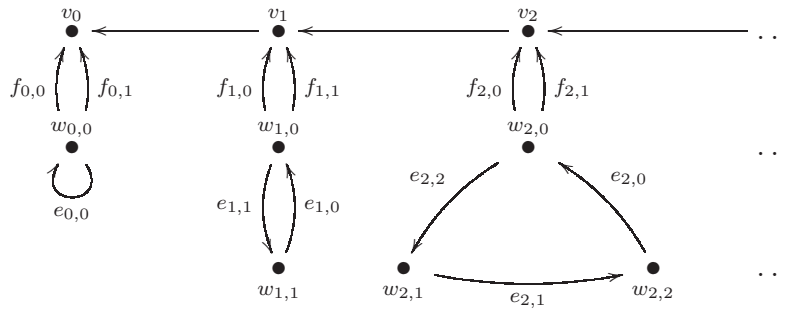
\[\xymatrix{\overset{v_0}\bullet&&\overset{v_1}\bullet\ar[ll]&&\overset{v_2}\bullet\ar[ll]&&\dots\ar[ll]\\
\overset{w_{0,0}}\bullet\ar@(dr,dl)^{e_{0,0}}\ar@/^/[u]^{f_{0,0}}\ar@/_/[u]_{f_{0,1}}
&&\overset{w_{1,0}}\bullet\ar@/^/[u]^{f_{1,0}}\ar@/_/[u]_{f_{1,1}}\ar@/_/[d]_{e_{1,1}}
&&\overset{w_{2,0}}\bullet\ar@/^/[u]^{{f}_{2,0}}\ar@/_/[u]_{{f}_{2,1}}
\ar@/_/[ld]_{e_{2,2}}&&\dots\\
&&\underset{w_{1,1}}\bullet\ar@/_/[u]_{e_{1,0}}&\underset{w_{2,1}}\bullet
\ar@/_/[rr]_{e_{2,1}}&&\underset{w_{2,2}}\bullet\ar@/_/[lu]_{e_{2,0}}&\dots
}
\]
Let $x\in E^\infty$.  If $x=\alpha\alpha\alpha...$ for some cycle
$\alpha$  with $r(\alpha)=w_{n,k}$ for  some $0\leq k\leq n$, then
the isotropy subgroup of $x$ in $\G_E$  is $A_{x}=(n+1)\Z$;
otherwise $A_x=\{0\}$. It is straightforward to check that the
isotropy subgroups vary continuously in the Fell topology. We claim
that the groupoid $C^*$-algebra $C^*(\G_E)$ has bounded trace  but
is not a Fell algebra.  To see this, by Theorems~\ref{thm-end}
and~\ref{thm-Cartan2},  we need to show that $\cR:=\G_E/\cA$ is
integrable but not Cartan.

We start by considering the following graph $F$ from \cite[\S8]{HaH}:
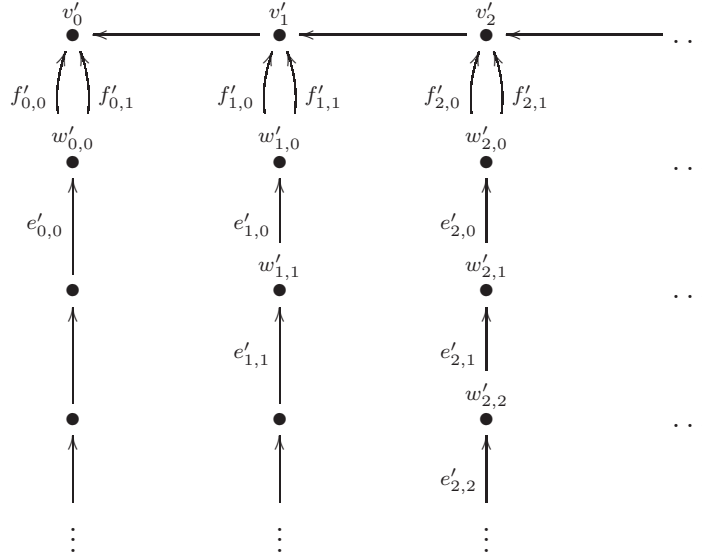
\[\xymatrix{\overset{v'_0}\bullet&&\overset{v'_1}\bullet\ar[ll]&&\overset{v'_2}\bullet\ar[ll]&&\dots\ar[ll]\\
\overset{w'_{0,0}}\bullet\ar@/^/[u]^{f'_{0,0}}\ar@/_/[u]_{f'_{0,1}}
&&\overset{w'_{1,0}}\bullet\ar@/^/[u]^{f'_{1,0}}\ar@/_/[u]_{f'_{1,1}}&&\overset{w'_{2,0}}\bullet
\ar@/^/[u]^{f'_{2,0}}\ar@/_/[u]_{f'_{2,1}}&&\dots\\
\bullet\ar[u]^{e'_{0,0}}&&\overset{w'_{1,1}}\bullet\ar[u]^{e'_{1,0}}&&\overset{w'_{2,1}}\bullet
\ar[u]^{e'_{2,0}}&&\dots\\
\bullet\ar[u] &&\bullet\ar[u]^{e'_{1,1}}&&\overset{w'_{2,2}}\bullet
\ar[u]^{e'_{2,1}}&&\dots\\
\vdots\ar[u]&&\vdots\ar[u] &&\vdots\ar[u]^{e'_{2,2}}}\] There are no
cycles in $F$, so $\G_F$ is a principal groupoid by
\cite[Proposition~8.1]{HaH}. By \cite{Robbie}, the groupoid $\G_F$
is integrable.  So $C^*(\G_F)$ has bounded trace by
\cite[Theorem~4.4]{CaH}. For $n\geq 0$, let $x^n$ be the unique
infinite path with range $v'_0$ which has $f'_{n,0}$ as an edge, let
$y^n$ be the  unique infinite path with range $v'_0$ which has
$f'_{n,1}$ as an edge, and let $z$ be the infinite path going
through each $v'_i$.  It is shown in \cite[Example~8.2]{HaH} that
the sequence $\{x^n\}$ converges $2$-times in $\G_F^{(0)}/\G_F$ to
$z$; the sequences in $\G_F$ witnessing this $2$-times convergence
are $\gamma_n^{(1)}=(x^n,0,x^n)\text{\ and \
}\gamma_n^{(2)}=(y^n,0,x^n)$. It follows that $\G_F$ is not a Cartan
groupoid by Lemma~\ref{lem-2times}.  Since $\G_F$ is principal,
$C^*(\G_F)$ is not a Fell algebra by \cite[Theorem~7.9]{C}.

Now consider the open subset  \[U=\bigcup_{i\geq
0}\big(Z(v'_i)\cup_{j\leq i} Z(w'_{i,j})\big)
\]
of $F^\infty$. Let $G$ be the groupoid obtained by restricting
$\G_F$ to $U$.  Then $G$ is a principal, integrable  groupoid which
is not Cartan (because the two sequences witnessing the $2$-times
convergence of $\{x^n\}$ in $\G_F$ are also in $G$).  Thus
$C^*(G)$ has bounded trace but is not a Fell algebra.

We claim that $G$ is isomorphic to $\cR=\G_E/\cA$. To see this,
first note that ``unwrapping'' cycles in $E^*$ sets up a bijection
$\phi$  between $E^*$ and the set of finite paths in $U$; similarly
``unwrapping'' cycles in $E^\infty$ sets up  a bijection $\psi$
between $E^\infty$  and the set of infinite paths in $U$.  If
$\alpha$ is a finite path in $E^*$ then
$Z(\phi(\alpha))=\psi(Z(\alpha))$. Since the cylinder sets form a
basis for the topology on $E^\infty$, $\psi:E^\infty\to U$ is a
homeomorphism.

Second,  fix $(x, k,y)\in \G_E$ so that  $x\sim_k y$. Then either
(1) $x$ and $y$ are of the form $x=\alpha\gamma\gamma...$,
$y=\beta\gamma\gamma....$ where $\gamma$ is a cycle with
$r(\gamma)=w_{n,0}$ for some $n\in\N$, and $\alpha,\beta\in E^*$
don't contain $\gamma$, or (2) both $x$ and $y$ do not contain
cycles.  In (1), $\psi(x)$ and $\psi(y)$ are shift-equivalent with
lag $|\alpha|-|\beta|$, and in (2) $\psi(x)$ and $\psi(y)$ are
shift-equivalent with lag $k$. Thus, since $G$ is principal, if $(x,
k,y)\in \G_E$ then there exists a unique $l_k$ such that $(\psi(x),
l_k,\psi(y))\in G$.

Finally, it is now straightforward to verify that \[\rho:\G_E\to G
\text{\ defined by\  }\rho((x,k,y))=(\psi(x), l_k,\psi(y))\] is a
groupoid homomorphism which is continuous, open and surjective, and
that $\rho$ induces a homeomorphism $\rho:\cR\to G$.  Thus $\cR$ is
an integrable groupoid which is not Cartan, and hence $C^*(\G_E)$
has bounded trace  but is not a Fell algebra by
Theorems~\ref{thm-end} and~\ref{thm-Cartan2}.

\end{example}

\begin{example}  Let $\G_E$ be the groupoid from Example~\ref{ex-1}.  Let $\cD_E$ be
the  associated  $\T$-groupoid over   $\hat\cA\rtimes\cR$ defined
by Muhly-Williams-Renault
 (see \eqref{eq-D}).  Note that the orbit space  $\G_E^{(0)}/\G_E$ is $T_1$, so  $C^*(\G_E)$ and $C^*(\cD_E;\hat\cA\rtimes\cR,\delta\times\alpha)$ are isomorphic by
 Proposition~\ref{isomo}. Thus $C^*(\cD_E;\hat\cA\rtimes\cR,\delta\times\alpha)$ has bounded trace but is not a Fell algebra, and hence  by Theorems~\ref{thm-hard}
 and~\ref{thm-Fell}, $\hat\cA\rtimes\cR$ is an integrable groupoid but is not Cartan.
\end{example}


\appendix\section{Corrections to the proof of Theorem~2.3 of \cite{MW90}\\ contributed by Dana P. Williams.}

Robert Hazlewood pointed out that there is a problem with the proof
of Theorem~3.2 in \cite{MW90}.  On the bottom of page~237, we assert
that we can find neighbourhoods $V_{0}$ and $V_{1}$ of $z$ such that
$\overline{V_{0}}\subseteq V_{1}$ with the property that\footnote{We
are retaining
  the notations of \cite{MW90} except we have dropped the
  fraktur font for groupoids and written $G$ in place of $\frak
G$ for clarity.  This is more of an issue in \cite{MW92} where
our readers have been frustrated trying to distinguish between $\frak G$,
$\frak S$ and $\frak E$ --- rather than between $G$, $S$ and $E$.}
\begin{equation}
  \label{eq:1}
  \overline{W_{1}}^{7}\overline{V_{1}}\; \overline{W_{1}}^{7} \setminus
  W_{0}V_{0}W_{0}\subseteq r^{-1}(\go\setminus N).
\end{equation}
Unfortunately, if
$\overline{V_{1}}$ is larger than $V_{0}$, then
we see no reason such neighbourhoods should exist.  In fact, we now
suspect that it is not
possible to find such neighbourhoods --- let alone via
a ``straightforward
compactness argument''.  However, \eqref{eq:1} does hold provided we
restrict to elements with source sufficiently close to $z$.\footnote{A
similar restriction was required in \cite{wiljfa81} --- the function
$f_{x}^{1}$ defined on the bottom of \cite[p.~61]{wiljfa81} is only
well-defined on $U_{0}$ (even though I failed to mention this).  This
is reflected in the statement of \cite[Lemma~4.4]{wiljfa81}.}  Namely,
we can prove the following.

\begin{lemma}\label{danalemma1}
  Given neighbourhoods $V_{0}$ and $V_{1}$ of $z$ in $G$ with $V_{0}$
  open and $V_{1}$ relatively compact, there is a compact neighbourhood $A$ of
  $z$ in $\go$ such that
  \begin{equation}
    \label{eq:2}
    (\overline{W_{1}}^{7}\overline {V_{1}} \setminus W_{0}V_{0})\cap
    G_{A}\subseteq r^{-1}(\go\setminus N),
  \end{equation}
where $G_{A}:=s^{-1}(A)$.
\end{lemma}
\begin{proof}
  If no such $A$ exists, we can let $\set{A_{n}}$ be a neighbourhood
  basis of $z$ with each $A_{n}$ compact and $A_{n+1}\subseteq A_{n}$.
  Then, by assumption, for each $n$
  we can find $\gamma_{n}\in G_{A_{n}}$ belonging
  to the closed set $r^{-1}(N)\cap
  (\overline{W_{1}}^{7}\overline{V_{1}} \setminus W_{0}V_{0})$.  Since
  $\overline{W_{1}}^{7}\overline{V_{1}}$ is compact, we can pass to a
  subsequence, relabel, and assume that $\gamma_{n}\to \gamma$.
  Notice that we must have $\gamma\in r^{-1}(N)\cap
  (\overline{W_{1}}^{7} \overline{V_{1}}\setminus
  W_{0}V_{0})$.
Since $s(\gamma_{n})\in A_{n}$ and $s(\gamma_{n})\to s(\gamma)$, we
must have $s(\gamma)=z$.  Since $\gamma\notin W_{0}V_{0}$, we have
$\gamma\in G_{z}\setminus W_{0}z$.  Since $F_{z}\subseteq W_{0}z$, our
construction of $F_{z}$ forces $r(\gamma)\notin N$.  But this is a
contradiction.  This completes the proof of the Lemma.
\end{proof}

Now, if $\gamma\in G_{A}$, then
\begin{equation}\label{eq:3}
  g^{(1)}(\gamma):=
  \begin{cases}
    g\bigl(r(\gamma)\bigr) &\text{if $\gamma\in \overline{W_{1}}^{7}
      \overline{V_{1}}\;\overline{W_{1}}^{7}$ and } \\
0&\text{if $\gamma\notin W_{0}V_{0}W_{0}$}
  \end{cases}
\end{equation}
is a well defined function on $G_{A}$.  Consequently \eqref{eq:3} defines
an element of $C_{c}(G_{A})$.  We can use the
Tietze-Extension Theorem to extend $g^{(1)}$ to an element of
$C_{c}(G)$ provided we keep in mind that \eqref{eq:3} holds only for
$\gamma\in G_{A}$.

Next, we must modify \cite[Lemma~2.8]{MW90} to hold only near $z$; specifically, we
have the following.

\begin{lemma}\label{danalemma2}
  With the choices above,
  \begin{equation*}
    g\bigl(r(\gamma)\bigr)g\bigl(r(\alpha)\bigr) b(\gamma\alpha^{-1})
    g^{(1)}(\alpha) =g^{(1)}(\gamma)g\bigl(r(\alpha)\bigr) g^{(1)}(\alpha)
  \end{equation*}
\emph{provided $\gamma,\alpha\in G_{A}$}.
\end{lemma}

Then with the given restriction on $\gamma$ and $\alpha$, the proof of
Lemma~\ref{danalemma2} goes through as written in \cite{MW90}.
Now it is straightforward to check that the rest of the proof of
\cite[Theorem~2.3]{MW90} goes through with the observation that (1)~we only need
consider the representations $L^{u}$ with $u$ close to $z$, and that
(2)~$L^{u}$ acts on $L^{2}(G_{u},\lambda_{u})$.  This allows us to
apply Lemma~\ref{danalemma2}  at the appropriate time.

\begin{remark}
The existence of the neighbourhoods $V_0$ and $V_1$ such that
\eqref{eq:1} holds is used in  the proof of \cite[Theorem~2.3]{MW92}
(see page~237 of \cite{MW92}) and in the proof of
\cite[Proposition~4.1]{CaH}; both results are saved by
Lemma~\ref{danalemma1} since we can restrict to elements with source
sufficiently close to $z$.
\end{remark}


 \end{document}